\documentclass[11pt,english,a4paper]{smfart}
\usepackage[english]{babel}
\numberwithin{equation}{section}

\usepackage{mathabx}
\usepackage{enumerate}
\usepackage{amsmath,amssymb, bm,amsfonts} 
\usepackage{bbm}

\usepackage[all]{xy}
\entrymodifiers={+!!<0pt,\fontdimen22\textfont2>}
\def\apl#1#2#3{#1:
\xymatrix{#2\ar[r]&#3}
}
\usepackage{ mathrsfs }

\newtheorem{theorem}{Theorem}[section]
\newtheorem{lemma}[theorem]{Lemma}

\newtheorem{corollary}[theorem]{Corollary}
\newtheorem{proposition}[theorem]{Proposition}

\newtheorem*{coro*}{Corollary}

{\theoremstyle{definition}}

{\theoremstyle{definition}\newtheorem{example}[theorem]{Example}}

\theoremstyle{definition}
\newtheorem{definition}[theorem]{Definition}
\newtheorem{question}[theorem]{Question}
\newtheorem{fact}[theorem]{Fact}
\newtheorem{claim}[theorem]{Claim}

{\theoremstyle{definition}\newtheorem{remark}[theorem]{Remark}}

\def\D{\ensuremath{\mathbb D}}

\def\T{\ensuremath{\mathbb T}}
\def\R{\ensuremath{\mathbb R}}
\def\Z{\ensuremath{\mathbb Z}}

\def\C{\ensuremath{\mathbb C}}
\def\Q{\ensuremath{\mathbb Q}}
\def\N{\ensuremath{\mathbb N}}

\newcommand{\pss}[2]{\ensuremath{{\langle #1,#2\rangle}}}

\newcommand{\wh}[1]{\widehat{#1}}

\newcommand{\ka}{Kazhdan}

\newcommand{\qq}{Q}

\newcommand{\set}[1]{\left\{#1\right\}}
\newcommand{\wrt}{with respect to}

\newcommand{\op}{operator}
\newcommand{\sep}{separable}
\newcommand{\js}{Jamison sequence}

\newcommand{\ld}{L^{2}(\T,\sigma )}
\newcommand{\la}[1]{\langle #1\rangle}
\newcommand{\pt}{Property (T)}
\newcommand{\kk}{k\ge 0}
\newcommand{\iz}{n\in\Z}
\newcommand{\nq}{n\in\qq}

\newcommand {\be}{\begin{equation}}
\newcommand {\ee}{\end{equation}}
\newcommand {\beq}{\begin{eqnarray*}}
\newcommand {\eeq}{\end{eqnarray*}}

\DeclareMathOperator{\ke}{Ker}
\renewcommand{\ker}{\ke}

\newcommand{\eve}{eigenvector}
\newcommand{\eva}{eigenvalue}

\author{Catalin Badea}
\address{Universit\'{e} de Lille, CNRS, UMR 8524 - Laboratoire Paul Painlev\'{e}, B\^atiment M2, 59655 Villeneuve d'Ascq Cedex, France}
\email{catalin.badea@univ-lille1.fr}
\author{Sophie Grivaux}
\address{CNRS,
Laboratoire Ami\'enois de Math\'{e}matique Fondamentale et Appliqu\'{e}e, UMR 7352,
Universit\'{e} de Picardie Jules Verne,
33 rue Saint-Leu,
80039 Amiens Cedex 1,
France. {\sl Current address}:  CNRS, Laboratoire Paul Painlev\'{e}, UMR 8524, Universit\'{e} de Lille, B\^atiment M2, 59655 Villeneuve d’Ascq Cedex, France}
\email{grivaux@math.univ-lille1.fr}
\date{November 1, 2017}

\thanks{This material is based upon work supported by EU IRSES grant AOS (PIRSES-GA-2012-318910), by the Labex CEMPI (ANR-11-LABX-0007-01), and also by the National Science
Foundation under Grant No. DMS-1440140 while the first author was in residence
at the Mathematical Sciences Research Institute in Berkeley, California,
during the Fall 2016 semester.
}

\title[Jamison and Kazhdan sets in $\mathbb{Z}$]{Sets of integers determined by operator-theoretical properties: Jamison and Kazhdan sets in the group $\mathbb{Z}$}
\alttitle{Ensembles d'entiers d\'etermin\'es par des propri\'et\'es op\'eratorielles : ensembles
de Jamison et de Kazhdan dans le groupe $\Z$.}

\begin{document}
\begin{abstract}
The aim of this partly expository paper is to present and discuss two classes of sets of integers (Jamison and Kazhdan sets) whose definition and/or properties are determined or inspired by operator-theoretical properties. Jamison sets first appeared in the study of the relationship between the growth of the sequence of norms of iterates of a bounded linear operator on a separable Banach space and the size of its unimodular point spectrum. Kazhdan subsets of $\Z$ are particular cases of Kazhdan sets in 
general topological groups, which are especially important as they appear 
in the definition of Property (T). This paper is also intended as a 
companion to the authors' paper \cite{BaGr}, which undertakes a study of 
Kazhdan subsets of some classical groups without Property (T). We present here 
in detail the case of the group $\Z$, which is one of the most natural 
examples of groups without Property (T), and which may be useful to 
build an intuition of some of the main results of \cite{BaGr}. Also, the 
proofs in the case of the group $\Z$ rely solely on tools from basic 
operator theory and harmonic analysis. Some crucial links between Jamison 
and Kazhdan sets in $\Z$ are exhibited, and many examples are given.
\end{abstract}

\begin{altabstract}
Nous pr\'esentons dans ce texte deux classes de sous-ensembles de $\Z$ (les
ensembles de Jamison et les ensembles de Kazhdan) dont la d\'efinition et/ou
les propri\'et\'es sont motiv\'ees et/ou inspir\'ees par des propri\'et\'es de nature op\'era\-torielle. Les ensembles de Jamison sont apparus dans l'\'etude des relations
entre, d'une part, la croissance des normes des it\'er\'es d'un op\'erateur born\'e
sur un espace de Banach s\'eparable et, d'autre part, la taille de son spectre
ponctuel unimodulaire. Les ensembles de Kazhdan dans $\Z$ sont 
des cas parti\-culiers d'ensembles de Kazhdan dans les groupes g\'en\'eraux; ces derniers sont
particuli\`erement importants puisqu'ils apparaissent dans la d\'efinition de la
Propri\'et\'e (T). Ce texte compl\'emente \'egalement notre article \cite{BaGr}, dans lequel
nous \'etudions les sous-ensembles de Kazhdan de certains groupes classiques
n'ayant pas la Propri\'et\'e (T). Nous pr\'esentons ici en d\'etail le cas du groupe $\Z$,
qui est l'un des exemples les plus \'el\'ementaires de groupes ne poss\'edant pas la
Propri\'et\'e (T), et qui peut \^etre utile pour appr\'ehender certains des r\'esultats
principaux de \cite{BaGr}. De plus, les preuves dans ce cas particulier reposent exclu\-sivement sur des outils de th\'eorie des op\'erateurs \'el\'ementaire et d'analyse
harmonique. Nous pr\'esentons \'egalement certains liens importants entre les
ensembles de Jamison et les ensembles de Kazhdan, et donnons de nombreux
exemples.
\end{altabstract}
\maketitle

\section{Introduction}

Our aim in this partly expository paper is to present a study of two 
particular classes of subsets of $\Z$, defined by operator-theoretic 
conditions: \emph{Jamison sets} and \emph{\ka\ sets} in $\Z$. Our 
study fits into the general framework of the investigation of sets of 
integers defined by properties coming from various fields of mathematics: 
combinatorics, dynamical systems and ergodic theory, harmonic  
analysis... or operator theory. Well-known examples of such classes of 
sets are recurrence sets for dynamical systems in various contexts, or pointwise convergence sets (see for instance the accounts \cite{Fur}, \cite{Fr} or \cite{Be2}, 
among many other works on the subject), van der Corput sets (see for instance 
\cite{BL} and the references therein), Bohr sets (\cite{Be1}, \cite{Katz}, \cite{HK}), Rajchman and Riesz sets, Sidon sets 
(\cite{HMP}), and (very) many others. In several of these settings, the original definition 
admits a reformulation in operator-theoretic terms. We consider here 
classes of sets originally defined in operator-theoretic terms, motivated 
either by natural problems in operator or representation theory, or by 
questions pertaining to the study of geometry of groups 
via representation theory. In the two cases on which we focus here -- 
namely Jamison sets and \ka\ sets in $\Z$ -- one of our aims is to obtain 
characterizations of such classes of sets in terms of measures, which are 
much more tractable than the initial definitions and allow a thorough 
study of the aforementioned classes.
\par\smallskip 
Jamison sets were first defined in operator theory as subsets of $\N$, in 
the study of the relationship between the growth of the norms of the 
iterates of a bounded linear operator on a separable complex Banach space 
and the size of its peripheral point spectrum.
\ka\ sets in $\Z$ are particular cases of \ka\ sets in general topological 
groups. This class of sets is especially important, as it 
is involved in the definition of the famous Property (T): a 
topological group has \emph{\ka's Property (T)}, or is a \emph{\ka\ 
group}, if it admits a compact \ka\ set. As any other  non-compact locally compact 
amenable group, $\Z$ does not have Property (T), and thus admits no 
compact (i.e.\ finite in this case) \ka\ set. Still, it makes sense to 
investigate the structure of \ka\ sets in groups which do not have 
Property (T), and several questions along these lines were proposed in 
\cite[Sec.\,7.12]{BdHV}. Some of them were solved in the authors' paper \cite{BaGr}, 
where a new characterization  of \ka\ sets in general topological groups 
was obtained which yielded a characterization of \ka\ sets in many classical 
groups. This approach also yielded an answer to a
question proposed by Shalom in \cite[Sec.\,7.12]{BdHV} concerning the 
links between equidistribution properties of sequences $(n_{k}\theta 
)_{k\ge 0}$ for irrational numbers $\theta $ and the fact that the set 
$\{n_{k}\,;\,k\ge 0\}$ is a \ka\ subset of $\Z$. Although the results of 
\cite{BaGr} hold in a much more general framework, understanding the case of 
the group $\Z$ proved to be a crucial step towards the general results. In 
particular, as we will see, Jamison sets play a key role in the proof of the results of 
\cite{BaGr} specialized to the group $\Z$.
\par\smallskip 
Thus, besides being a survey paper about Jamison and \ka\ sets in $\Z$, 
the present article is also intended as a companion to \cite{BaGr}, in which 
we present independent proofs of some of the results of \cite{BaGr} in the 
case of the group $\Z$. The proofs in this particular setting rely only on 
classical tools from harmonic analysis, whereas the general proofs of 
\cite{BaGr} involve more sophisticated tools  from abstract harmonic 
analysis, ergodic theory and representation theory. Also, the case of the 
group $\Z$ is extremely useful to build a good intuition of what is likely 
to happen in the study of \ka\ sets in arbitrary groups. The reader 
familiar with \cite{BaGr}
will realize easily enough that the proof of Theorem 2.3 of \cite{BaGr} relies on the same ideas as the proof of Theorem \ref{Theorem 2} below: convolution of measures become tensor product of representations in the proof of \cite[Th.~2.3]{BaGr}, continuous measures are replaced by weakly mixing representations, and although Jamison sets do not appear in the proof of \cite[Th.~2.3]{BaGr}, the dichotomy between Case 1 and Case 2 reflects the dichotomy between Jamison and non-Jamison sets in the case of $\Z$.
\par\smallskip
The paper is organized as follows: we first present in 
Section \ref{sec:2} below the definition and the main properties of Jamison 
sets in the original setting of \cite{BG1} and \cite{BG2}, namely that of 
subsets of $\N$. We introduce Jamison subsets of $\Z$ in Section \ref{sec:3}, and briefly indicate how the known results about Jamison subsets 
of $\N$ extend to this setting. We also give in this section several 
examples of Jamison and non-Jamison subsets of $\Z$. In Section 
\ref{sec:4}, we consider \ka\ subsets of $\Z$ as a particular case of the 
more general notion of \ka\ sets in topological groups, and we 
characterize  them in several ways. The two most useful characterizations 
are given in terms of Fourier coefficients of probability measures on the 
unit circle. While the first of these two characterizations (Theorem 
\ref{Theorem 1}) follows from rather standard arguments, the 
second one (Theorem \ref{Theorem 2}) is far less obvious, and actually 
forms the core of our study of \ka\ subsets of $\Z$. After giving several 
examples of \ka\ and non-\ka\ subsets of $\Z$, we prove Theorem 
\ref{Theorem 2} in Section \ref{Section 6}. Section \ref{Section 7} deals with the 
link between equidistribution properties modulo one and \ka\ sets in $\Z$. 
In particular, we re-obtain the answer of \cite{BaGr} to the above mentioned question of
Shalom (\cite[Sec.\,7.12]{BdHV}).

 \section{Jamison sets: the origins}\label{sec:2}
 Let us first fix some notation, which will be used throughout the paper. We will denote by $\T$ the unit circle in $\C$ and by $\D$ the open unit disk: $\T = \{\lambda \in \C : |\lambda| = 1\}$  and $\D = \{\lambda \in \C : |\lambda| < 1\}$. If
 $T \in \mathcal{B}(X)$ is a bounded linear operator on a complex Banach space $X$, we denote by $\sigma_p(T) $ the \emph{point spectrum} of $T$ (i.e. the set of eigenvalues of $T$). The set $\sigma_p(T)\cap\T$ of eigenvalues of $T$ of modulus $1$ is called the \emph{unimodular point spectrum} of $T$. We 
 say that $T\in \mathcal{B}(X)$ is \emph{power bounded} if $\sup_{n\ge 0}  \|T^n\| $ is finite. Observe that the spectrum of a power bounded operator is always included in the closed unit disc. Lastly,  a set is said to be
 countable if it stands in one-to-one correspondence with a subset of $\N$.
 
 \subsection{Definition of \js s} 
 The study of \js s of positive integers originated in the following question: how is the growth of the sequence $(\|T^n\|)_{n\ge 0}$ of the norms of the iterates of $T$ influenced by the size of the (unimodular) point spectrum of $T$ and by the geometry of the ambient space $X$?
 \par\smallskip
 A basic result in this direction is due to Jamison \cite{J}, who proved in $1965$ that if $X$ is separable and $T$ is power bounded, the unimodular point spectrum $\sigma_p(T)\cap \T$  of $T$ is countable.   
 Jamison's result is optimal in several senses: suitable unitary diagonal operators on a (separable or non-separable) Hilbert space show that $\sigma_p(T)\cap \T$ may be any countable set under the hypotheses of the result, and that one cannot drop the assumption that $X$ be separable. Also, nothing can be said in general about the size of the set of eigenvalues of a power-bounded operator which are not on the unit circle. For instance, the backward shift $B$ on $\ell^2(\C)$ given by $B(x_0,x_1,\cdots) = (x_1,x_2, \cdots)$, $x=(x_{n})_{n\ge 0}\in\ell^{2}(\C)$, has any $\lambda \in \D$ as an eigenvalue but has empty unimodular point spectrum. Jamison \cite{J} gave an example of an operator of spectral radius one on a separable Hilbert space which has uncountably many unimodular eigenvalues.
  \par\smallskip
 In view of the result of Jamison, it was natural to investigate whether, given a strictly increasing sequence $(n_k)_{k\ge 0}$ of integers, the condition
 $\sup_{k\ge 0}\|T^{n_k}\| <+\infty$ (where $T$ is a bounded \op\ on a separable Banach space $X$) implies that $\sigma_p(T)\cap\T$ is  countable.
 When $X$ is a general Banach space and $\sigma_p(T)\cap \T$ is assumed to be uncountable, Ransford \cite{R} proved that the norms $||T^{n}||$ of the iterates of $T$ tend to infinity along a subset of density $1$ of $\N$. This last result was complemented in the paper \cite{RR}, where Ransford and Roginskaya showed that the norms $\|T^n\|$ do not necessarily tend to infinity under such assumptions, even for separable Banach spaces. More precisely, they constructed, for each sequence $(n_k)_{k\ge 0}$ of integers such that $n_k$ divides $n_{k+1}$ for each $k\ge 0$ and $(n_k)_{k\ge 0}$ grows fast enough, a separable Banach space $X$ and an operator $T$ on $X$ with uncountable unimodular point spectrum such that the sequence $(\|T^{n_k}\|)_{k\ge 0}$ is nonetheless bounded. 
 This motivates the following definition.
 
 \begin{definition}[\cite{RR}]
 Let $X$ be a separable Banach space. Let $(n_k)_{k\ge 0}$ be a strictly increasing sequence of positive integers and let $T$ be a bounded linear operator on $X$. We say that $T$ is \emph{partially power bounded} with respect to the sequence $(n_k)_{k\ge 0}$ if $\sup_{k\ge 0}\|T^{n_k}\| $ is finite. 
 \end{definition}
 
The question to know for which sequences of integers partial power-boundedness implies countable unimodular point spectrum was investigated further in \cite{BG1} and \cite{BG2}, where the following definition was introduced:

\begin{definition}
Let $(n_{k})_{k\geq 0}$ be a strictly increasing sequence of
positive integers. We say that $(n_{k})_{k\geq 0}$ is a \emph{Jamison sequence}
if for any separable Banach space $X$ and any bounded
\op\ $T$ on $X$, $\sigma _{p}(T)\cap\T$ is  countable
as soon as $T$ is partially power-bounded \wrt\ $(n_{k})_{k\ge 0}$.
\end{definition}

\subsection{A characterization of \js s} We shall see that 
being a Jamison sequence or not depends  not only on 
the growth of the sequence but also on its
arithmetical properties. Our aim now is to present a complete characterization of Jamison
sequences which was obtained in \cite{BG2}. In order to formulate this characterization we need to introduce a distance on
the unit circle $\T$ associated to a given sequence $(n_{k})_{k\geq 0}$
with $n_{0}=1$: for $\lambda ,\mu \in\T$, let us define
$$d_{(n_{k})}(\lambda ,\mu )=\sup_{k\geq 0}|\lambda ^{n_{k}}-\mu ^{n_{k}}|.$$
This distance is used in \cite{RR} as well as in \cite{BG1,BG2}
 for the construction of non-\js s. 
The assumption that $n_{0}=1$ comes into play in order to ensure that $d_{(n_{k})}$ is indeed a distance on $\T$. It could be replaced by the weaker assumption that the set $\{n_{k} \textrm{ ; } k\ge 0\}$ generates the group $\Z$, but we will for simplicity's sake keep the hypothesis that $n_{0}$ be equal to $1$. Observe also that the property of being a \js\ remains unchanged by the addition or the removal of finitely many terms of the sequence. Assuming that the first term is equal to $1$ is thus no real restriction.

\begin{theorem}[\cite{BG2}]\label{th2}
Let $(n_{k})_{k\geq 0}$ be a strictly increasing sequence of positive integers
with $n_{0}=1$. The following assertions are equivalent:
\begin{itemize}
\item[(1)] the sequence $(n_{k})_{k\ge 0}$ is a \js;

\item[(2)] for every uncountable subset $K$ of $\T$, the metric space $(K,d_{(n_{k})})$
is non-separable;

\item[(3)] for every uncountable subset $K$ of $\T$, there exists  $\varepsilon >0$ such that
$K$ contains an uncountable $\varepsilon $-separated family for the
distance $d_{(n_{k})}$;

\item[(4)] there exists  $\varepsilon >0$ such that every uncountable subset $K$ of $\T$
contains an uncountable $\varepsilon $-separated family for the
distance $d_{(n_{k})}$;

\item[(5)] there exists $\varepsilon >0$ such that any
two distinct points $\lambda $ and $\mu $ in $\T$ are $\varepsilon
$-separated for the distance $d_{(n_{k})}$:
$$\textrm{for every } \lambda \not=\mu, \quad
\sup_{k\geq 0}|\lambda ^{n_{k}}-\mu ^{n_{k}}|\geq \varepsilon .$$
\end{itemize}
\end{theorem}

We obtain in particular the following  characterization of \js s, which is the most useful one for practical as well as theoretical applications.

\begin{theorem}[\cite{BG2}]\label{th0}
Let $(n_{k})_{k\geq 0}$ be a strictly increasing sequence of
integers with $n_{0}=1$. The following assertions are equivalent:
\begin{itemize}
\item[(1)] $(n_{k})_{k\geq 0}$ is a \js;

\item[(2)] there exists a positive real number $\varepsilon $ such that
for every $\lambda \in\T\setminus \set{1}$, $$\sup_{k\geq 0}|\lambda ^{n_{k}}-1|
\geq \varepsilon .$$
\end{itemize}
\end{theorem}

Indeed, condition $(2)$
in Theorem \ref{th0} can be reformulated in terms of the distance $d_{(n_{k})}$ by saying that
distinct points of $\T$ are uniformly separated for $d_{(n_{k})}$.

\begin{remark}
Let $\theta \in \R$. Let us denote by 
$$\|\theta\| := \min\left(\{\theta\}, 1-\{\theta\}\right) = \inf \{|\theta-n| : n \in \Z\}$$ the distance of $\theta$ to the nearest integer ($\{\theta\}$ denotes here the fractional part of $\theta$). Then 
$$ 4\|\theta\| \le |e^{2i\pi \theta} - 1| \le 2\pi \|\theta\| \quad\textrm{ for every }\theta\in\R.$$
The condition (2) of Theorem \ref{th0} can be thus expressed as follows: there exists a positive real number $\varepsilon' $ such that for every $\theta\in (0,1/2]$ we have 
$$\sup_{k\geq 0}\|n_{k}\theta\| \geq \varepsilon' .$$ 
\end{remark}

The most difficult part of the proof of Theorem \ref{th2} is that of the implication $(1) \Rightarrow (2)$. Assuming that a sequence $(n_{k})_{k\ge 0}$ does not satisfy (2), and that there exists an uncountable subset $K$ of $\T$ such that $(K, d_{(n_{k})})$ is separable, we have to construct an operator on a separable Banach space which is partially power bounded with respect to  $(n_k)_{k\geq 0}$ and has uncountable unimodular point spectrum.
Starting from the Hilbert space of complex sequences
$$H =\{x=(x_{j})_{j\geq 0} \textrm{ ; }
||x||=\left(\sum_{j\geq 0} \frac
{|x_{j}|^{2}}{j^{2}+1}\right)^{\frac {1}{2}}<+\infty \},$$ we consider the backward shift
$S $ on $H$. For every $\lambda\in\T$,
$e_{\lambda }=(\lambda , \lambda ^{2}, \lambda ^{3}, \ldots)$ is an eigenvector of $S$ associated to the eigenvalue $\lambda$.
A new norm on this space $H$ is defined by setting
$$\left|\left|x\right|\right|_{new}=
\max\left( \left|\left|x\right|\right|, \,\sup_{j\geq 0} \; 2^{-(j+1)}
\sup_{n_{k_{0}}, \ldots, n_{k_{j}}\in\{n_{k}\}}
 \left|\left|
\prod_{l=0}^{j} (S^{n_{k_{l}}}-I)x\right|\right|\right)$$
and we set $X_{new}=\{x\in X \textrm{ ; } ||x||_{new}<+\infty \}$.
Notice that
 $e_{\lambda }$ belongs to $X_{new}$ for every
$\lambda \in\T$. Indeed, we have
$$||e_{\lambda }||_{new}=\max\left(1,
\sup_{j\geq 0} \; 2^{-(j+1)}
\sup_{n_{k_{0}}, \ldots, n_{k_{j}}\in\{n_{k}\}}
\prod_{l=0}^{j} |\lambda ^{n_{k_{l}}}-1|
\right)\,||e_{\lambda }||=||e_{\lambda }||.$$
If $T$ denotes the \op\ induced by $S$ on
$X_{new}$,  we have
$\sup_{k\geq 0}\left|\left|T^{n_{k}}\right|\right|_{new}\leq
3$. Thus, with 
respect to the new norm, $T$ is partially power bounded with respect to 
$(n_k)_{k\geq 0}$.
The proof goes on by establishing the existence of a constant $C>0$ such that for every $\lambda ,\mu \in\T$,
$$
||e_{\lambda }-e_{\mu }||_{new}\leq C\, d_{(n_{k})}(\lambda ,\mu ).
$$
Our assumption that $(K,d_{(n_{k})})$ is separable then implies that the space $$X_{new}^{K}=\overline{\textrm{span}}^{||\,.\,||_{new}}[{e_{\lambda} \textrm{ ; } \lambda\in K}]$$ is
separable, and the unimodular point spectrum of the \op\ induced by $T$ on 
$X_{new}^{K}$ is
uncountable, as it contains $K$.
\par\smallskip
The proof of the implication $(3)\Rightarrow (5)$ in Theorem \ref{th2} shows that 
assertions $(3)$ and $(4)$ actually admit ``infinite'' instead of ``uncountable''
versions which are equivalent to $(n_{k})_{k\ge 0}$ being a \js:

\begin{corollary}\label{cor0}
The assertions of Theorem \ref{th2} are also equivalent to the
following:
\begin{itemize}
\item[(3')] for every uncountable subset $K$ of $\T$, there exists $\varepsilon >0$ such that
$K$ contains an infinite $\varepsilon $-separated family for the
distance $d_{(n_{k})}$;

\item[(4')] there exists $\varepsilon >0$ such that every uncountable subset $K$ of $\T$
contains an infinite $\varepsilon $-separated family for the
distance $d_{(n_{k})}$.
\end{itemize}
\end{corollary}

For a fixed sequence $(n_k)_{k\ge 0}$ with $n_0=1$, we denote, for each  $\varepsilon>0$, by $\Lambda_{\varepsilon}$ the subset of $\T$ defined by
\begin{equation}\label{eqajoutee}
 \Lambda _{\varepsilon }=\{\lambda \in\T \textrm{ : } \sup_{k\geq 0}|\lambda ^{n_{k}}-1|
<\varepsilon \}.
\end{equation}

Theorem \ref{th2} and \cite[Prop. 2.1]{RR} imply the following result:

\begin{corollary}\label{cor00}
The assertions of Theorem \ref{th2} are also equivalent to:
\begin{itemize}
\item[(6)] there exists  $\varepsilon >0$ such that $\Lambda _{\varepsilon }$
is  countable.
\end{itemize}
\end{corollary}

Thus, as soon as each one of the sets
$\Lambda _{\varepsilon }$ has at least two elements, 
all of them are automatically
uncountable.

\subsection{Universal Jamison spaces}
The assertion that $(n_{k})_{k\geq 0}$ is not a Jamison sequence means that there exists a \sep\ Banach space $X$ and a bounded operator $T\in\mathcal{B}(X)$ such that $\sup_{k\geq
 0}||T^{n_{k}}||<+\infty$ and $\sigma _{p}(T)\cap\T$ is uncountable. But the space $X$ may very well be extremely complicated: as was briefly sketched above, in the original proof of Theorem \ref{th2} the space is obtained by a rather involved renorming argument. See also \cite[Chapter 6]{BM} for a slightly modified description of this construction. 
We now introduce the following definition.
 
 \begin{definition}
 Let $X$ be a separable infinite-dimensional complex Banach space. We say that $X$ is a \emph{universal Jamison space} if the following property holds true: for any increasing sequence of positive integers $(n_k)_{k\ge 0}$ which is not a Jamison sequence, there exists
a bounded
\op\ $T$ on
 $X$
which is partially power-bounded \wrt\ $(n_{k})_{k\ge 0}$ and whose unimodular point spectrum $\sigma _{p}(T)\cap\T$ is uncountable.
\end{definition}

Eisner and Grivaux proved in \cite{EG} that a separable infinite-dimensional complex Hilbert space is a universal Jamison space. The proof of \cite{EG} is completely explicit: the operators $T$ with $\sup_{k\geq 0}||T^{n_{k}}||<+\infty $  and $\sigma  _{p}(T)\cap\T$ uncountable constructed there are perturbations by a weighted backward shift on $\ell_2(\N)$ (endowed with its canonical basis) of a diagonal \op\ with unimodular diagonal coefficients. For other Banach spaces, the following results are known.

\begin{theorem}[\cite{EG}]\label{thm:EiGr}
The space $\ell_p(\N)$ is a universal Jamison space  for every $1\le p<+\infty$. 
\end{theorem}

Using the ideas of \cite{EG}, this was generalized by Devinck in \cite{Dev} to a larger class of Banach spaces.
 
\begin{theorem}[\cite{Dev}]\label{thm:Dev}
Any separable Banach space which admits an unconditional Schauder decomposition is a universal Jamison space.
\end{theorem} 

We recall here that a separable Banach space $X$ admits an \emph{unconditional Schauder decomposition} if there exists a sequence $(X_{\ell})_{\ell\ge1}$ of closed subspaces of $X$ (different from $\{0\}$) such that any vector $x$ of $X$ can be written in a unique way as an unconditionally convergent series $\sum_{\ell\ge1}x_{\ell}$, where $x_{\ell}$ belongs to $X_{\ell}$ for all  $\ell\ge 1$. There are many examples of Banach spaces which admit an unconditional Schauder decomposition, like spaces possessing an unconditional Schauder basis. Additional examples are $C([0,1])$, the James space and all spaces containing a copy of $c_0(\N)$. On the other hand (see \cite{Dev}), a hereditarily indecomposable Banach space is not a universal Jamison space since the unimodular point spectrum of any operator on such a space is countable (see \cite{Maurey} for details)). 

\subsection{Examples of Jamison sequences} 
Using the characterization of \js s given by Theorem \ref{th0}, it is easy to obtain many examples of such sequences. We present here some of them, following \cite{BG2}. We have chosen to present the complete proofs whenever they are not too long.

\begin{example}
 Let $(n_k)_{k\ge 0}$ be the sequence of all positive integers: $n_k = k+1$ for every $k\ge 0$. Then $(n_k)_{k\ge 0}$ is a \js.
\end{example}

\begin{proof}
This is the classical result of Jamison.
It is also a folklore result in character theory (or a simple exercise about the geometric aspects of complex numbers) that 
$(n_k)_{k\ge 0}$
satisfies  condition (2) of Theorem \ref{th0}.  It can even be proved that $\varepsilon = \sqrt{3}$ is the best Jamison constant for $(n_{k})_{k\ge 0}$ in the following sense:  if $z\in \T$ is such that $|z^{k+1} - 1| < \sqrt{3}$ for all $k\ge 0$, then $z =1$, and this constant $\sqrt{3}$ is optimal. 
Here is a possible proof of this statement. Let $z\in\T\setminus\{ 1\}$. If $z$ is not a root of unity, then the subgroup generated by $z$ is dense in the unit circle. In particular, there exists $k\ge 0$ such that  $|z^{k+1} - 1| \ge \sqrt{3}$.  If $z$ is a root of unity of even minimal order $ 2n$ then $z^n = -1$. The last case that we need to consider is when $z$ is a root of unity of odd minimal order $2n+1$ for some $n\ge 1$.  It is then enough to consider the case when $z = e^{2i\pi/(2n+1)}$. We have $\frac{2n\pi}{2n+1}  = \pi - \frac{\pi}{2n+1}\ge \frac{2\pi}{3}$. Therefore 
$$ \left|\exp({\frac{2in\pi}{2n+1}}) - 1\right|^2 = 2 - 2 \cos(\frac{2n\pi}{2n+1}) \ge 2 - 2 \cos (\frac{2\pi}{3}) = 3.$$
The example $\lambda = e^{2i\pi/3}$ shows that the constant $\sqrt{3}$ appearing in this statement cannot be improved. 
\end{proof}

\begin{example}[\cite{RR}]\label{ex1}
Let $(n_{k})_{k\geq 0}$ be a strictly  increasing sequence of positive integers
such that
$$\sup_{k\geq 0}\,
\frac {n_{k+1}}{n_{k}}<+\infty .$$ Then $(n_{k})_{k\ge 0}$ is a \js. For instance, the sequence  defined by $n_k = 2^k$, $k\ge 0$, is a \js.
\end{example}

\begin{proof}
Without loss of generality, we can assume that $n_{0}=1$. 
Let $$\kappa =\sup_{k\geq 0} \frac {n_{k+1}}{n_{k}}\cdot$$ We are going to show that condition (2) of Theorem \ref{th0}
is satisfied with $\varepsilon =2\,\sin\frac {\pi}{2\kappa }$. Let $\lambda\in\T\setminus \{1\}$, which we write as
 $\lambda =e^{i\theta }$, $0<\theta<2\pi$. Without loss of generality we suppose
that $0<\theta \leq \pi$. Since $(n_{k})_{k\ge 0}$ is a strictly increasing
sequence
with $n_{0}=1$, there exists a (unique) integer $k \ge 0$ such that
$\frac {\pi}{2n_{k+1}}<\frac {\theta }{2}\leq \frac {\pi}{2n_{k}}$.
Hence $\frac {\pi}{2\kappa }<\frac {n_{k}\theta }{2}\leq \frac {\pi}{2}$
so that for this special choice of $k$, $$|\lambda ^{n_{k}}-1|=2|\sin\frac {n_{k}\theta }{2}|
\geq 2\sin\frac {\pi}{2\kappa }\cdot$$ This proves our statement.
\end{proof}

\begin{example}\label{ex4}
Let $(n_{k})_{k\ge 0}$ be a strictly increasing sequence of integers such that the set $\{n_{k} \textrm{ ; } k\ge 0\}$ contains intervals of arbitrary length.
Then $(n_{k})_{k\ge 0}$ is a \js.
\end{example}

\begin{proof}
We suppose as usual that $n_0=1$. Let $\delta _{0}$ be such that $0<\delta_0<2$. We are going to show that for every $\lambda =e^{i\theta }
\in\T\setminus\{1\}$, $0<\theta \leq \pi$, with $\theta/\pi$ irrational, there exists $k\ge 0$ such
that $\lambda ^{n_{k}}=e^{in_{k}\theta }$ lies outside the arc $I_{\delta_0}=\{\lambda\in\T\textrm{ ; }|\lambda-1|<\delta_0\}$. 
For each $p\ge 1$, let $[N_{p}, \ldots, N_{p}+p-1]$
be an interval of length $p$ contained in the set $\{n_{k}\textrm{ ; } k\ge 0\}$. 
Since $\theta/\pi$ is irrational, there exists an integer $p\ge 1$ such that the numbers $\lambda^{n}$, $0\le n<p$, form a $(2-\delta_0)$-net of $\T$ (for the distance on $\T$ induced by the absolute value on $\C$). In particular there exists $0\le n<p$ such that $|\lambda^{n}+\lambda^{-N_p}|<2-\delta_0$, i.e. $|\lambda^{N_p+n}+1|<2-\delta_0$. Then $\lambda^{N_p+n}$ does not belong to the arc $I_{\delta_0}$.
Since the integer $N_{p}+n$ has the form $n_{k}$ for some $k\ge 1$, and since the reasoning above holds true for any $0<\delta_0<2$, we obtain that $\sup_{k\geq 0}|\lambda ^{n_{k}}-1|=2$ for every
 $\lambda=e^{i\theta}$ with $\theta/\pi$ irrational. Hence the set $\Lambda_2$ (see (\ref{eqajoutee}) for the definition) is countable, and by Corollary \ref{cor00} we obtain that $(n_{k})_{k\ge 0}$ is a \js.
\end{proof}

\begin{example}(\cite{R})\label{ex5}
If $(n_{k})_{k\ge 0}$ is a strictly increasing sequence of integers such that the set $\{n_{k}\;\,;\,k\ge 0\}$ has positive upper density,
 $(n_{k})_{k\ge 0}$ is a \js.
\end{example}

\begin{proof}
We suppose again that $n_0=1$. Since $D=\{n_{k} \textrm{ ; } k\ge 0\}$ has positive upper density, there exists
$\delta \in(0,1)$
and a strictly increasing sequence $(N_{r})_{r\ge 1}$ of integers such that $\#[1, N_{r}]\cap D\geq \delta
\,N_{r}$  for
every $r\geq 1$. Fix $0<\theta_{0}<\delta\pi$, and let $\gamma>0$ be such that $\theta_{0}+\gamma<\delta\pi$. For any $\lambda\in\T$ with $\lambda=e^{i\theta}$, $0<\theta<\pi$, $\theta/\pi$ irrational, we have by the uniform distribution of the sequence $(\lambda^{n})_{n\ge 0}$ in $\T$ that $$\#\{1\le n\le N_{r}\textrm{ ; } \lambda^{n}\in (e^{-i\theta_0}, e^{i\theta_0})\}
\le N_{r}\frac{2\theta_{0}+2\gamma}{2\pi}$$ for $r$ sufficiently large.
Hence $$\#\{1\le n\le N_{r}\textrm{ ; } \lambda^{n}\not\in (e^{-i\theta_0}, e^{i\theta_0})\}
\ge N_{r}\frac{2\pi-2\theta_0-2\gamma}{2\pi}=N_{r}(1-\frac{\theta_0+\gamma}{\pi})> N_{r}(1-\delta)$$ 
for all $r$ sufficiently large. It follows from the fact that $\#[1, N_{r}]\cap D\geq \delta
\,N_{r}$
that for all such $r$ there exists $1\le n\le N_{r}$ belonging to $D$ such that $\lambda^{n}$ does not belong to the arc $(e^{-i\theta_0}, e^{i\theta_0})$. We have thus proved that $\sup_{k\ge 0}|\lambda^{n_{k}}-1|\ge |e^{i\theta_0}-1|>0$ for every $\lambda=e^{i\theta}$ with $\theta/\pi$ irrational. By Corollary \ref{cor00}, $(n_{k})_{k\ge 0}$ is a \js.
\end{proof}

Condition $(6)$ of Corollary \ref{cor00}, which is weaker than Condition $(5)$ in Theorem \ref{th2},
is already used in \cite{RR} and \cite{BG1} in order to obtain \js s. We use it again in the
following example (cf. \cite{BG2}).
Recall that  a set $\Sigma = (\sigma_k)_{k\ge 0}$ of real numbers is said to be
 \emph{dense modulo} $1$ if the set
$\Sigma + \Z = \{\sigma_k + n : k \ge 0, n\in \Z\}$
is dense in $\R$. For any $\eta > 0$, the set $\Sigma$ is said
to be $\eta$-\emph{dense modulo} $1$ if the set
$\Sigma+\Z$ intersects every open sub-interval of $\R$ of length greater than $ \eta$.

\begin{example}\label{ex5bis'}
 Let $(n_k)_{k\ge 0}$ be a strictly increasing sequence of integers. If
there exists a number $0<\eta < 1$
such that the set
$$D_{\eta} = \set{\theta\in \R : (n_k\theta)_{k\ge 0} \text{ is not } \eta\text{-dense modulo } 1}$$
is countable, then $(n_k)_{k\ge 0}$ is a \js.
\end{example}

\begin{proof}
We can suppose that $n_0 = 1$. Let $I$ be a closed
sub-arc of $\T$, not containing the point $1$, of length $2\pi\eta<2\pi$. If $\lambda= e^{2i\pi\theta }$ is such that 
$\theta$ does not
belong to $D_{\eta}$, there exists a $k\ge 0$ such that $\lambda ^{n_{k}}$
belongs to $I$. Hence there exists $\varepsilon >0$ such that  $\sup_{k\geq 0}|\lambda ^{n_{k}}-1|\geq \varepsilon $ for every
such $\lambda $. The set $\Lambda _{\varepsilon }$ defined in (\ref{eqajoutee})
is thus countable for this particular choice of $\varepsilon $, and hence $(n_{k})_{k\ge 0}$ is a
\js.
\end{proof}

In particular, if $(n_k\theta)_{k\ge 0}$ is dense modulo $1$ for every irrational
$\theta$ (for instance, if $(n_k\theta)_{k\ge 0}$ is uniformly distributed modulo $1$ for every irrational $\theta$), then $(n_k)_{k\ge 0}$ is a \js.

\subsection{Sequences which are not \js s}
The fact that all the sets $\Lambda _{\varepsilon }$, $\varepsilon>0$, are uncountable as soon as all of them are
 non-trivial 
 greatly simplifies the task of exhibiting non-\js s,
in the sense that we only have to construct one non-trivial point belonging to each one of the sets $\Lambda_\varepsilon$ instead of uncountably many
ones. We recall below the following example of non-\js s:

\begin{example}[\cite{BG1}]\label{ex6}
Let $(n_{k})_{k\ge 0}$ be a strictly increasing sequence of integers such that $\frac {n_{k+1}}{n_{k}}$ tends to
infinity as $k$ tends to infinity. Then $(n_{k})_{k\ge 0}$ is not a \js.
\end{example}

\begin{proof}
A first proof is given in \cite{BG1}, and a simplified new proof in \cite{BG2}. 
\end{proof}

In view of Example \ref{ex4} above, it is natural to ask whether a
sequence $(n_{k})_{k\ge 0} $ containing arithmetic progressions of
arbitrary length is necessarily a \js. It is not so, as shown in the next example which was pointed out
to us independently by Evgeny Abakumov and Vladimir M\"{u}ller.

\begin{example}\label{ex6bis}
If $(n_k)_{k\ge 0}$ is the strictly increasing sequence of integers defined by the condition $$\{n_{k}\,;\,k\ge 0\}=\bigcup_{r\geq 1}\{(r!)^{2}, 2(r!)^{2},
\ldots, r(r!)^{2}\},$$ then $(n_{k})_{k\ge 0}$ is not a \js, but it contains arithmetic
progressions of arbitrary length.
\end{example}

\begin{proof}
Since $n_{0}=1$, it suffices by Theorem \ref{th0} to exhibit for each $\varepsilon >0$ a number $\lambda \in\T
\setminus\{1\}$ such that $\sup_{k\geq 0}|\lambda ^{n_{k}}-1|<\varepsilon
$. Let $r_{0}$ be an integer such that $\frac {((r_{0}+1)!)^{2}}
{r_{0}(r_{0}!)^{2}}>
\frac {2\pi}{
\varepsilon }$, and let $\lambda =e^{i\theta }$, where
$\theta =\frac {2\pi}{((r_{0}+1)!)^{2}}$.
Then $\lambda ^{((r_{0}+1)!)^{2}}=1$ so that $\lambda ^{n_{k}}=1$ for
every $k$ such that $n_{k}\geq (r_{0}+1)!^{2}$.
 Now if $n_{k}< ((r_{0}+1)!)^{2}$, i.e. $n_{k}\leq r_{0}(r_{0}!)^{2}$,
$\lambda^{n_{k}} =e^{{i n_{k}\theta }}$ and
$$|\lambda ^{n_{k}}-1|\leq 2\pi \frac {n_{k}}{((r_{0}+1)!)^{2}}\leq
2\pi \frac {r_{0}(r_{0}!)^{2}}{((r_{0}+1)!)^{2}}<\varepsilon ,$$ so that
$\sup_{k\geq 0}|\lambda ^{n_{k}}-1|<\varepsilon
$.
\end{proof}

The same kind of proof shows that
the following conditions suffice for a sequence to be non-Jamison:

\begin{example}\label{ex7}
Let $(n_{k})_{k\geq 0}$ be a strictly increasing sequence of integers such that  $n_{k}$
divides $n_{k+1}$ for every $k\geq 0$, and $\overline{\textrm{lim}}_{\,k \to +\infty}\frac {n_{k+1}}{n_{k}}=+\infty $. Then $(n_{k})_{k\ge 0}$
is not a \js.
\end{example}

\begin{proof}
Again we can assume that $n_0 = 1$. Fix $\varepsilon >0$, and let $k_{0}\ge 1$ be
such that $\frac {n_{k_{0}+1}}{n_{k_{0}}}>
\frac {2\pi}{
\varepsilon }$. If $\lambda =e^{i\theta }$, where
$\theta =\frac {2\pi}{n_{k_{0}+1}}$, then
the divisibility assumption on the integers $n_{k}$ implies that $\lambda ^{n_{k}}=1$
for every $k>k_{0}$. Also,
$\lambda^{n_{k}} =e^{{i n_{k}\theta }}$ for every $0\le k\leq k_{0}$, and we obtain in the same
way as in Example \ref{ex6bis} above that
$\sup_{k\geq 0}|\lambda ^{n_{k}}-1|<\varepsilon
$.
\end{proof}

As a corollary, we obtain our last example in this section:

\begin{example}\label{cor1}
Let $(n_{k})_{k\geq 0}$
 be a strictly increasing sequence of integers such that  $n_{k}$
divides $n_{k+1}$ for every $k\geq 0$. Then $(n_{k})_{k\ge 0}$
is  a \js\ if and only if 
 $\sup_{k\ge 0} \frac {n_{k+1}}{n_{k}} $ is finite.
\end{example}
 
The notion of \js\ can be extended, in a natural way, to that of  Jamison subset of $\Z$. As Jamison sets will play an important role in our forthcoming study of \ka\ subsets of $\Z$, we give briefly in Section \ref{sec:3} below the definitions and results which will be needed.

 \section{Jamison sets in $\Z$}\label{sec:3}
 Taking into account that power bounded operators are nothing but bounded representations of the semigroup $\N$, the following definition of Jamison sets in $\Z$ is quite natural. This definition first appeared in \cite{Dev2}, in the more general context of the study of Jamison subsets of certain abelian groups.
 
 \begin{definition}\label{DefJamZ}
 A subset $Q$ of $\Z$ is said to be a \emph{Jamison set} in $\Z$ if it has the following property: whenever $T$ is a bounded invertible operator on a complex separable Banach space $X$ such that 
$\sup_{n\in Q}||T^{n}||$ is finite, the unimodular point spectrum $\sigma _{p}(T)\cap \T$ of $T$ is countable.
\end{definition}
 
Jamison subsets $Q$ of $\Z$ can be characterized in exactly the same way as \js s were characterized in Theorem \ref{th0} above.
 We will need to assume in what follows that $1$ belongs to $\qq$ (or, more generally, that the subgroup generated by $Q$ is the whole of $\Z$).
 
 \begin{theorem}\label{th0bis}
 Let $\qq$ be a subset of $\Z$ containing the point $1$. The following assertions are equivalent:
\begin{itemize}
\item[(1)] $Q$ is a Jamison set in $\Z$;

\item[(2)] there exists a positive real number $\varepsilon $ such that
for every $\lambda \in\T\setminus \set{1}$, $$\sup_{n\in Q}|\lambda ^{n}-1|
\geq \varepsilon .$$
\end{itemize}
 \end{theorem}
 
 \begin{proof}
 The proof of Theorem \ref{th0bis} is very similar to that of Theorem \ref{th0}, as given in \cite{BG2}. Let $\qq$ be a subset of $\Z$ containing $1$. The easy implication is the following: if we have, for some $\varepsilon>0$, $\sup_{\nq}|\lambda ^{n}-1|\ge \varepsilon $ whenever $\lambda \in\T\setminus \{1\}$, then $\qq$ is a Jamison set in $\Z$. Indeed, let $X$ be a \sep\ Banach space, and let $T\in \mathcal{B}(X)$ be
an invertible operator which is partially power-bounded \wrt\ $Q$. We set $M=\sup_{{\nq}} ||T^{n}||$. If $\lambda $
and $\mu $ are two different \eva s of $T$, let $e_{\lambda }$ and $e_{\mu }$
be two associated \eve s of $T$ with $||e_{\lambda }||=||e_{\mu }||=1$. Then $T^pe_{\lambda} = \lambda^pe_{\lambda}$ for every $p\in \Z$.
We have $$||e_{\lambda }-e_{\mu }||
\geq \frac {1}{M+1}|\lambda ^{n_{k}}-\mu ^{n_{k}}|$$ for every $k\geq 0$.
Since $\lambda \not=\mu $, $$||e_{\lambda }-e_{\mu }||
\geq \frac {\varepsilon }{M+1}$$ by our assumption, and the separability
of $X$ implies that $\sigma _{p}(T) \cap \T$ is countable.
\par\smallskip
The proof of the converse implication follows the same method as in \cite{BG2}. We present only the (very minor) necessary changes. Starting from the Hilbert space
$$H =\{(x_{j})_{j\in \Z} \textrm{ ; }
||x||=\left(\sum_{j\in \Z} \frac
{|x_{j}|^{2}}{j^{2}+1}\right)^{\frac {1}{2}}<+\infty \},$$ we consider the bilateral shift
$S $ on $H$
and its \eve s $e_{\lambda }=(\ldots, \lambda^{-2},\lambda^{-1}, 1,\lambda , \lambda ^{2}, \ldots)$, $\lambda \in\T$. 
We define a new norm on this space $H$ by setting
$$\left|\left|x\right|\right|_{new}=
\max\left( \left|\left|x\right|\right|, \,\sup_{j\ge 0} \; 2^{-(j+1)}
\sup_{n_{{0}}, \ldots, n_{{j}}\in Q}
 \left|\left|
\prod_{l=0}^{j} (S^{n_{l}}-I)x\right|\right|\right),$$
and define $X_{new}=\{x\in X \textrm{ ; } ||x||_{new}<+\infty \}$. Then
 $e_{\lambda }$ belongs to $X_{new}$ for every
$\lambda \in\T$, and the invertible \op\ $T$ induced by $S$ on a suitable closed subspace of
$X_{new}$ produces the desired example of a partially power-bounded operator \wrt\ $Q$ on a separable Banach space
with uncountable unimodular point spectrum.
 \end{proof}
 
 A constant $\varepsilon >0$ for which assertion (2) of Theorem \ref{th0bis} above holds true is called a \emph{Jamison constant} for $\qq$.
The same proof as that of \cite[Th.\,2.1]{BG2} shows the following result:

\begin{theorem}\label{The0Ter}
 If $\qq$ is a subset of $\Z$ containing $1$ which is not a Jamison set, there exists for every $\varepsilon >0$ an (uncountable) perfect compact set of elements $\lambda \in\T$ such that $\sup_{\nq}|\lambda ^{n}-1|<\varepsilon $.
\end{theorem}

An important consequence of Theorem \ref{The0Ter} is the following: suppose that $\qq$ is a subset of $\Z$ containing $1$ which has the following property: there exists $\varepsilon >0$ such that the set of $\lambda\in \T$ such that $\sup_{\nq}|\lambda ^{n}-1|<\varepsilon $ is countable. Then $\qq$ is a Jamison set. Thanks to this, one can show easily, in the same way as in \cite[Cor.\,2.3]{BG1} or Example \ref{ex5bis'} above, the following result:

\begin{example}\label{ex8}
Let $(n_{k})_{\kk}$ be a sequence of elements of $\Z$ such that $n_{0}=1$ and $(n_{k}\theta )_{\kk}$ is uniformly distributed modulo $1$ for every $\theta \in\R\setminus\Q$. Then $Q=\{n_{k}\,;\,\kk\}$ is a Jamison set in $\Z$.
\end{example}

We will come back to this example in Section \ref{Section 6}.

 \section{Kazhdan sets in $\Z$}\label{sec:4}
 
Before starting to discuss the particular case of the group $\Z$, let us say some words about \ka\ sets in general topological groups.
 
\subsection{Kazhdan sets in topological groups} We start by recalling briefly the relevant definitions, and refer the reader to \cite{BdHV} for more information. All the Hilbert spaces that we consider in this paper are complex, and all the unitary representations of topological groups are assumed to be strongly continuous.
 
\begin{definition}\label{Definition 1.1}
  Let $\qq$ be a subset of a topological group $G$, and let $\varepsilon >0$.
\begin{enumerate}
 \item [(i)] If $\pi $ is a unitary representation of $G$ on a Hilbert space $H$, a vector $x\in H$ is  said to be
\emph{$(\qq,\varepsilon )$-invariant} for $\pi $ if $\sup_{g\in\qq}||\pi (g)x-x||<\varepsilon ||x||$. In particular $x$ is non-zero. A \emph{$G$-invariant} vector for $\pi $ is a vector $x\in H$ such that $\pi (g)x=x$ for every $g\in G$;
\item[(ii)] the pair $(\qq,\varepsilon )$ is called a \emph{\ka\ pair in $G$} if any unitary representation $\pi $ of $G$ on a Hilbert space $H$ which admits a $(\qq,\varepsilon )$-invariant vector has a non-zero $G$-invariant vector;
\item[(iii)] the set $\qq$ is \emph{a \ka\ set in $G$} if there exists 
$\varepsilon >0$ such that $(\qq,\varepsilon )$ is a \ka\ pair in $G$. Such an $\varepsilon >0$ is called a \emph{\ka\ constant} for $Q$;
\item[(iv)] the group $G$ \emph{has \pt}\ if it admits a compact \ka\ set. 
\end{enumerate}
 \end{definition}

\pt\  is a rigidity property of topological groups which has been introduced by \ka\ \cite{Ka} in order to show that certain lattices in locally compact groups are finitely generated. It has striking applications to many subjects, as shown in the authoritative monograph \cite{BdHV} by Bekka, de la Harpe and Valette. Typical examples of groups with Property (T) are the groups $SL_{n}(K)$ for $n\ge 3$, where $K$ is a local field. As a lattice in a locally compact group has Property (T) if and only if the group itself has it, $SL_{n}(\Z)$ has Property (T) for $n\ge 3$. Property (T) and amenability are two opposite properties, and the only locally compact amenable groups with \pt\ are the compact ones.
\par\smallskip
It is not an easy task to exhibit explicit \ka\ pairs in groups with \pt, as pointed out by Serre and de la Harpe and Valette \cite{BdHV}. Discrete groups with Property (T) are finitely generated, and \ka\ sets in such groups can be completely described: these are the generating subsets of the group \cite[Prop.~1.3.2]{BdHV}. See the references \cite{Sha2}, \cite{Ka}, \cite{Be} or \cite{Neuh} for some examples of constructions of explicit \ka\ sets in groups with \pt. When the group does not have \pt, it is also an interesting and difficult problem to exhibit ``small'' --although non-compact-- \ka\ subsets of the group (we remind that the pair $(G,\sqrt{2})$ is always a \ka\ pair in a topological group $G$, see \cite[Prop.~1.1.5]{BdHV}). This problem is mentioned in \cite[Sec.~7.12]{BdHV}, where, more precisely, the following two questions are stated. The second question is attributed in \cite[Sec.~7.12]{BdHV} to Y. Shalom.

\begin{question}[\mbox{\cite[Sec.~7.12]{BdHV}}]\label{Que2}
 Given a non-compact amenable group $G$, is it possible to characterize the \ka\ sets in $G$? What are the \ka\ sets in the groups $\Z^{d}$ or $\R^{d}$, $d\ge 1$?
\end{question}

\begin{question}[Shalom, \mbox{\cite[Sec.~7.12]{BdHV}}]\label{Que4}
Given a sequence $(n_{k})_{k\ge 0}$ of elements of $\Z$,
 what are the relations between the equi\-distribution properties modulo $1$ of the sequences $(n_{k }\theta)_{k\ge 0}$, $\theta\in \R\setminus \Q$, and the fact that $\{n_{k}\textrm{ ; } k\ge 0\}$ is a \ka\ set in $\Z$?
\end{question}

In the paper \cite{BaGr}, we answered Question \ref{Que4} by showing that if $(n_{k})_{\kk}$ is a sequence of elements of $\Z$ such that $(n_{k}\theta )_{\kk}$ is uniformly distributed modulo $1$ for every irrational number $\theta $, then $\{n_{k}\,;\,\kk\}$ is a \ka\ set in $\Z$ as soon as the subgroup it generates is $\Z$ itself. This result is a particular case of a much more general statement (\cite[Th.~2.1]{BaGr}), valid in any second countable locally compact Moore group $G$: if $(g_{k})_{k\ge 0}$ is a sequence of elements of $G$ satisfying a suitable equidistribution condition which generalizes that of Question \ref{Que4}, $\{g_{k}\,;\,k\ge 0\}$ is a \ka\ set in $G$ as soon as it generates $G$. Concerning Question \ref{Que2}, we study in \cite[Sec.~6]{BaGr} \ka\ sets in locally compact abelian groups, and provide a complete characterization of \ka\ sets in such groups, as well as in the Heisenberg groups, and the group $\textrm{Aff}_+(\R)$ of orientation-preserving homeomorphisms of $\R$.
\par\medskip 
The proofs of our main results in \cite{BaGr} rely on a new criterion for a subset $\qq$ of a topological group $G$ to admit a ``small perturbation'' which is a \ka\ set in $G$ (\cite[Th.~2.3]{BaGr}). As a particular case, we obtain the following necessary and sufficient condition for subsets of a locally compact group which generate the group to be \ka\ sets:

\begin{theorem}[\mbox{\cite[Th.~2.5]{BaGr}}]
 Let $G$ be a locally compact group, and let $\qq$ be a subset of $G$ which generates $G$. Then $\qq$ is a \ka\ set in $G$ if and only if the following property holds true: there exists $\varepsilon >0$ such that any unitary representation $\pi $ of $G$ on a Hilbert space $H$ with a $(\qq,\varepsilon )$-invariant vector has a finite-dimensional subrepresentation.
\end{theorem}

  As the results of \cite{BaGr} hold in a very general framework, their proofs involve some abstract tools: a key ingredient is an abstract version of the classical Wiener Theorem (\cite[Th.~3.7]{BaGr}) which, given a unitary representation $\pi $ of an arbitrary group $G$ on a Hilbert space $H$, yields an explicit expression for the quantities $m_{G}(|\pss{\pi (\,.\,)x}{y}|^{2})$, where $x$ and $y$ are two vectors of $H$. The notation $m_{G}$ represents here the unique invariant mean on the space $W\!AP(G)$ of weakly almost-periodic functions of $G$. The proof of Theorem~2.5 in \cite{BaGr} proceeds by contradiction, and uses this general version of the Wiener Theorem to show that certain infinite tensor products of unitary representations are weakly mixing. The spaces supporting these representations are incomplete tensor products of Hilbert spaces, which were first constructed by von Neumann in \cite{vN}, and whose study was later on taken up by Guichardet in \cite{Gui}.
 
 \subsection{Characterizing \ka\ sets in $\Z$}\label{Section 3} 
 The definition of a \ka\ set in $\Z$ can immediately be reformulated in the following way:

\begin{fact}\label{Fact 2.1}
 Let $\qq$ be a subset of $\Z$. The following assertions are equivalent:
\begin{enumerate}
 \item [(a)] $\qq$ is a \ka\ set in $\Z$;
\item[(b)] there exists $\varepsilon >0$ such that any unitary operator $U$ acting on a complex separable Hilbert space $H$ satisfies the following property: if there exists a vector $x\in H$ with $||x||=1$ such that $\sup_{\nq}||U^{n}x-x||<\varepsilon $, then there exists a non-zero vector $y\in H$ such that $Uy=y$ (i.e.\ $1$ is an eigenvalue of $U$).
\end{enumerate}
\end{fact}

As any cyclic unitary operator lives on a separable Hilbert space, it suffices to consider in (b) separable Hilbert spaces. Suppose indeed
that (b) holds true, and let $U$ be a unitary \op\ acting on a complex (possibly non-separable) Hilbert space $H$. Decomposing $U$ as a direct sum of cyclic unitary \op s, we can write $H$ as $H=\oplus_{i\in I}H_{i}$, where $H_{i}$ is separable for every $i\in I$, on which $U$ acts as $U=\oplus_{i\in I}U_{i}$, $U_{i}$ being a unitary \op\ on $H_{i}$ for every $i\in I$. Suppose that $x=\oplus_{i\in I}x_{i}\in H$ is such that $||x||=1$ and $\sup_{n\in Q}||U^{n}x-x||<\varepsilon/2$ (where $\varepsilon>0$ is given by (b) above). This means that $\sum_{i\in I}||x_{i}||^{2}=1$ and $\sup_{n\in Q}\sum_{i\in I}||U_{i}^{n}x_{i}-x_{i}||^{2}<\varepsilon^{2}/4$. Let $I_{0}$ be a finite subset of $I$ (depending of course on $x$) such that $\sum_{i\in I_{0}}||x_{i}||^{2}>1-\varepsilon^{2}/16$. Then $$\sum_{i\in I_{0}}||U_{i}^{n}x_{i}-x_{i}||^{2}\le \sum_{i\in I}||U_{i}^{n}x_{i}-x_{i}||^{2}+4\sum_{i\in I\setminus I_{0}}||x_{i}||^{2}\quad \textrm{for every }n\in\Z,$$ so that $$\sup_{n\in Q} \sum_{i\in I_{0}}||U_{i}^{n}x_{i}-x_{i}||^{2}<\varepsilon^{2}/2<\frac{\varepsilon^{2}}{2(1-\varepsilon^{2}/16)}\sum_{i\in I_{0}}||x_{i}||^{2}<\varepsilon^{2}\sum_{i\in I_{0}}||x_{i}||^{2}.$$ Hence we get that
$$\sup_{n\in Q}\left(\sum_{i\in I_{0}}||U_{i}^{n}x_{i}-x_{i}||^{2}\right)^{\frac{1}{2}}<\varepsilon
\left(\sum_{i\in I_{0}}||x_{i}||^{2}\right)^{\frac{1}{2}}.$$ Since $\oplus_{i\in I_{0}}U_{i}$ is a unitary \op\ on the separable Hilbert space $\oplus_{i\in I_{0}}H_{i}$, (b) applies, and $\oplus_{i\in I_{0}}U_{i}$ has a non-zero fixed point. This implies that $U$ itself has a non-zero fixed point. We have thus proved that $Q$ is a \ka\ set in $\Z$ (this proof is actually a mere rewriting of the well-known observation that if a topological group $G$ is second-countable, $Q$ is a \ka\ subset of $G$ as soon as there exists $\varepsilon>0$ such that property (ii) of Definition \ref{Definition 1.1} holds true for unitary representations of $G$ on \emph{separable} Hilbert spaces).
\par\smallskip

Applying property (b) of Fact \ref{Fact 2.1} to multiplication operators and using the spectral theorem for unitary operators, we first derive a characterization of \ka\ subsets of $\Z$ in terms of Fourier coefficients of probability measures on the unit circle.

\begin{theorem}\label{Theorem 1}
 Let $\qq$ be a subset of $\Z$. The following assertions are equivalent:
\begin{enumerate}
 \item [\emph{(1)}] $\qq$ is a \ka\ subset of $\Z$;
\item[\emph{(2)}] there exists $\varepsilon >0$ such that the following property holds true:
any probability measure $\sigma $ on $\T$ such that $\sup_{\nq}|\wh{\sigma }(n)-1|<\varepsilon $ satisfies $\sigma (\{1\})>0$.
\end{enumerate}
\end{theorem}

\begin{proof}
We first prove the implication $(1)\Longrightarrow(2)$. Let $\delta >0$ be a \ka\ constant for the set $\qq$, and let $\sigma $ be a probability measure on $\T$ such that $\sup_{\nq}|\wh{\sigma }(n)-1|<\delta ^{2}/2$.
We denote by $M_{\sigma }$ the multiplication operator by the independent variable $\lambda $ on the space $\ld$ and by $\textbf{1}$ the constant function equal to $1$. Then $||M_{\sigma }^{n}\textbf{1}-\textbf{1}||^{2}=2(1-\Re e\,\wh{\sigma }(n))$ for every $\iz$, so that $\sup_{\nq}||M_{\sigma }^{n}\textbf{1}-\textbf{1}||<\delta $. Since $M_{\sigma }$ is a unitary operator and $\textbf{1}$ a unit vector in $\ld$, it follows from (b) of Fact\,\ref{Fact 2.1} that there exists a non-zero  function $f\in \ld$ such that $M_{\sigma }f=f$, i.e.\ $(\lambda -1)f(\lambda )=0$ in $\ld$. Hence $\sigma (\{1\})>0$, which proves (2).
\par\medskip 
Let us now prove the implication $(2)\Longrightarrow (1)$. Let $U$ be a unitary operator on a separable complex Hilbert space $H$. By the spectral theorem for normal operators, there exists a finite or infinite sequence $(\sigma _{i})_{i\in I}$ of probability measures on $\T$ such that $U$ is similar, via an invertible isometry, to the direct sum operator $M=\bigoplus_{i\in I}M_{\sigma _{i}}$ acting on $K=\bigoplus_{i\in I}L^{2}(\T,\sigma _{i})$. Let $\apl{V}{H}{K}$ be an invertible isometry such that  $U=V^{*}MV$. Let now $\varepsilon >0$ be a constant satisfying assumption (2) of Theorem \ref{Theorem 1}, and suppose that $x\in H$ is a unit vector such that $\sup_{\nq}||U^{n}x-x||<\varepsilon $. Setting $Vx=f=\bigoplus_{i\in I}f_{i}$, with $f_{i}\in L^{2}(\T,\sigma _{i})$ for each $i\in I$, we have 
$||f||^{2}=\sum_{i\in I}||f_{i}||^{2}=1$ and 
\[
\bigl|\bigl|M^{n}f-f\bigr|\bigr|^{2}=\sum_{i\in I}\,\bigl|\bigl|M_{\sigma_{i}}^{n}f_{i}-f_{i}\bigr|\bigr|^{2}=\int_{\T}\bigl|\lambda ^{n}-1\bigr|^{2}\sum_{i\in I}\,\bigl|f_{i}(\lambda )\bigr|^{2}d\sigma _{i}(\lambda )\quad \textrm{for every}\ \iz.
\]
Let $\sigma $ be the positive measure on $\T$ defined by $$d\sigma(\lambda) =\sum_{i\in I}|f_{i}(\lambda)|^{2}d\sigma _{i}(\lambda), \quad\lambda\in \T.$$
Since $||f||^{2}=\sum_{i\in I}||f_{i}||^{2}=1$, $\sigma $ is a probability measure. It satisfies 
$||M^{n}f-f||^{2}=2(1-\Re e\,\wh{\sigma }(n))$ for every $\iz$, so that 
\[
\sup_{\nq}|1-\wh{\sigma }(n)|^{2}=2
\sup_{\nq}\,(1-\Re e\,\wh{\sigma }(n))=\sup_{\nq}||M^{n}f-f||^{2}=\sup_{\nq} ||U^{n}x-x||^{2}<\varepsilon^{2} .
\]
It follows that $\sigma (\{1\})>0$, so that $\sigma _{i_{0}}(\{1\})>0$ for some $i_{0}\in I$. Set 
$g_{i_{0}}=\textbf{1}_{\{1\}}$ (the characteristic function of the set $\{1\}$), $g_{i}=0$ for $i\in I\setminus \{i_{0}\}$, and $g=\bigoplus_{i\in I}g_{i}$. Then $g$ is a non-zero vector of $K$ with $Mg=g$. Since $U=V^{*}MV$, it follows that $y=V^{*}g$ is a non-zero vector of $H$ which satisfies $Uy=y$. Property (b) of Fact \ref{Fact 2.1} is proved, so $\qq$ is a \ka\ set in $\Z$. 
\end{proof}

\begin{remark}\label{rmk:47}
 The above proof yields the following quantitative statements:
 if $(Q,\delta)$ is a \ka\ pair for $\Z$, then assertion (2) of Theorem \ref{Theorem 1} holds true with $\varepsilon=\delta^{2}/2$;
conversely, if assertion (2) of Theorem \ref{Theorem 1} holds true with constant $\varepsilon$, then $(Q,\varepsilon)$ is a \ka\ pair for $\Z$.
\end{remark}

As we will use it repeatedly to exhibit non-\ka\ sets in $\Z$, we state separately the following obvious reformulation of Theorem \ref{Theorem 1}:
\begin{corollary}\label{Cor7}
 Let $\qq$ be a subset of $\Z$. The following assertions are equivalent:
\begin{enumerate}
\item[\emph{(1)}] $\qq$ is not a \ka\ set in $\Z$;
\item[\emph{(2)}] for any $\varepsilon >0$, there exists a probability measure $\sigma $ on $\T$ such that $\sigma (\{1\})=0$ and $\sup_{\nq}|\wh{\sigma }(n)-1|<\varepsilon $.
\end{enumerate}
\end{corollary}

As a consequence of Theorem \ref{Theorem 1}, we now show that whenever $\qq$ is a \ka\ set in $\Z$, property (b) of Fact \ref{Fact 2.1} above can be extended to arbitrary contractions on complex Hilbert spaces. Recall that if $T$ is  a bounded operator on a complex Hilbert space $H$, we denote by $\sigma (T)$ the spectrum of $T$, and by $\sigma _{p}(T)$ the point spectrum of $T$ (i.e.\ the set of eigenvalues of $T$). A contraction on $H$ is a bounded operator $T$ on $H$ with $||T||\le 1$. If $T$ is a contraction on $H$, we set $T_{n}=T^{n}$ for $n\ge 0$ and $T_{n}=T^{*n}$ for $n<0$. For every  $x\in H$, the sequence $(\la{T_{n}x,x})_{n\in \Z}$ is positive definite
(see \cite{FoNa}), so that there exists by Bochner's theorem  a positive measure $\mu _{x}$ on $\T$ which satisfies $\wh{\mu }_{x}(n)=\la{T_{n}x,x}$ for every $n\in\Z$. If $||x||=1$, $\mu _{x}$ is a probability measure on $\T$. Contractions (and, more generally, power-bounded operators) on a Hilbert space satisfy the von Neumann mean ergodic theorem: for every $x\in H$,
\[
\left|\left|\dfrac{1}{N}\sum_{j=0}^{N-1} T^{j}x-P_{\ker(T-\textrm{I})}x\right|\right|\!\xymatrix@C=17pt{\ar[r]&}\!0\quad\textrm{as}\quad N\!\!\xymatrix@C=17pt{\ar[r]&}\!\!+\infty ,
\]
where $P_{\ker(T-\textrm{I})}$ denotes the orthogonal projection of $H$ on the space $\ker(T-\textrm{I})$ of fixed points for $T$.

\begin{theorem}\label{The8}
 Let $\qq$ be a subset of $\Z$. The following assertions are equivalent:
\begin{enumerate}
 \item[\emph{(1)}] $\qq$ is a \ka\ set in $\Z$;
\item[\emph{(2)}] there exists $\varepsilon >0$ such that for any contraction $T$ on a complex Hilbert space $H$, the following holds true:
if there exists a vector $x\in H$ with $||x||=1$ such that $\sup_{\nq}|1-\la {T_{n}x,x}|<\varepsilon $, then $1$ belongs to $\sigma _{p}(T)$. 
\end{enumerate}
\end{theorem}

\begin{proof}
 The implication $(2)\Longrightarrow (1)$ follows immediately from Fact \ref{Fact 2.1}, observing that if $U$ is a unitary operator on $H$ and $x\in H$ is a unit vector such that $\sup_{\nq}||U^{n}x-x||<\varepsilon $, then $\sup_{\nq}|\la{U^{n}x,x}-\la{x,x}|<\varepsilon $, i.e.\ $\sup_{\nq}|\pss{U_{n}x}{x}-1|<\varepsilon $. 
\par\smallskip
It remains to prove that $(1)$ implies $ (2)$.
Let $\qq$ be a \ka\ set in $\Z$, and fix $\varepsilon >0$ satisfying property (2) of Theorem \ref{Theorem 1}. Let $T\in\mathcal{B}(H)$ be a contraction on a Hilbert space $H$, let $x$ be a unit vector in $H$ with $\sup_{\nq}|1-\la {T_{n}x,x}|<\varepsilon $, and let $\mu _{x}$ be a probability measure on $\T$ which satisfies 
$\wh{\mu }_{x}(n)=\la{T_{n}x,x}$ for every $n\in\Z$. Our assumption implies that $\sup_{\nq}|1-\wh{\mu }_{x}(n)|<\varepsilon $, from which it follows that $\mu _{x}(\{1\})>0$. Since 
\begin{align*}
\int_{\T}\Bigl(\dfrac{1}{N}\sum_{j=0}^{N-1} \lambda ^{j}\Bigr)d\mu _{x}(\lambda )&\!\xymatrix@C=17pt{\ar[r]&}\!\mu _{x}(\{1\})\quad\textrm{as}\quad N\!\!\xymatrix@C=17pt{\ar[r]&}\!\!+\infty , \intertext{and}
\int_{\T}\Bigl(\dfrac{1}{N}\sum_{j=0}^{N-1} \lambda ^{j}\Bigr)d\mu _{x}(\lambda )&=\dfrac{1}{N}\sum_{j=0}^{N-1}\,\la{T^{j}x,x},
\end{align*}
it follows from the mean ergodic theorem that $\mu _{x}(\{1\})=\pss{P_{\ker(T-\textrm{I})}x}{x}$. Since $\mu _{x}(\{1\})>0$, $\ker (T-\textrm{I})$ is non-zero, which means exactly that $1$ belongs to $\sigma _{p}(T)$.
\end{proof}

\begin{remark}
 The above proof gives the following relationships between the constants:
 if $(Q,\delta)$ is a \ka\ pair for $\Z$, then assertion (2) of Theorem \ref{The8} holds true with $\varepsilon=\delta^{2}/2$;
if assertion (2) of Theorem \ref{The8} holds true with constant $\varepsilon$, then $(Q,\varepsilon)$ is a \ka\ pair for $\Z$.
\end{remark}

 \begin{remark}
 Using unitary dilations, one can give a different proof of the implication $(1)\Longrightarrow (2)$ of Theorem \ref{The8}. Suppose that $(\qq,\delta)$ is a \ka\ pair in $\Z$. Let $T\in\mathcal{B}(H)$ be a contraction on a Hilbert space $H$, and let $h$ be a unit vector in $H$ with $\sup_{\nq}|1-\la {T_{n}h,h}|<\delta^2/2$. Let $U$ be the minimal unitary dilation of $T$ (see \cite{FoNa} for definitions), acting on a larger Hilbert space $K=H\oplus (K\ominus H)$ endowed with an inner product $\langle \cdot,\cdot\rangle_K$. Then $\langle U^n(h,0), (h,0)\rangle_K = \langle T_nh,h\rangle$ for every $n\in\Z$, and therefore we obtain
 \begin{eqnarray*}
 \|U^{n}(h,0) - (h,0)\|^2 & = & 2(1 - \Re e \langle U^n(h,0), (h,0)\rangle_K )\\
     & = & 2(1 - \Re e \langle T_nh,h\rangle)\le  2|1-\la {T_{n}h,h}|<\delta^2
 \end{eqnarray*}
 for all $n\in Q$.
 Hence $\sup_{n\in Q} \|U^{n}(h,0) - (h,0)\| < \delta$, and by Fact \ref{Fact 2.1} the point 1 belongs to the point spectrum of the unitary \op\ $U$. Since $U$ is the minimal unitary dilation of $T$, we also have (see \cite[Chapter 2]{FoNa}) that $1 $ belongs to $\sigma_p(T)$. 
 \end{remark}
 
\par\smallskip
The following characterization of \ka\ sets will be very important in the sequel. Rather surprisingly, it shows that the point $1$ does not play a special role in the description of \ka\ sets (either as an eigenvalue of the operators involved in the statement of Theorem \ref{The8} or as the point mass of the measures $\sigma $ appearing in Theorem \ref{Theorem 1}).

 \begin{theorem}\label{Theorem 2}
 Let $\qq$ be a subset of $\Z$ which contains the point $1$. The following assertions are equi\-valent:
\begin{enumerate}
 \item [\emph{(1)}] $\qq$ is a \ka\  subset of $Z$;
\item[\emph{(2)}] there exists $\varepsilon >0$ such that the following property holds true:
any probability measure $\sigma $ on $\T$ such that $\sup_{\nq}|\wh{\sigma }(n)-1|<\varepsilon $ has a discrete part.
\end{enumerate}
\end{theorem}

Theorem \ref{Theorem 2} can obviously be reformulated as follows:

\begin{corollary}\label{Cor11}
 Let $\qq$ be a subset of $\Z$ containing $1$. Then $\qq$ is not a \ka\ set in $\Z$ if and only if there exists for any $\varepsilon >0$ a continuous probability measure $\sigma $ on $\T$ such that 
$$\sup_{\nq}|\wh{\sigma }(n)-1|<\varepsilon .$$
\end{corollary}

We postpone the proof of Theorem \ref{Theorem 2} to the forthcoming Section \ref{Section 6} of the paper.

\section{Examples of Kazhdan sets in $\Z$}
As Jamison sets are a crucial tool for the proof of Theorem \ref{Theorem 2}, we begin this section by studying the links between the classes of Jamison and \ka\ sets in $\Z$. 

\subsection{Kazhdan versus Jamison sets in $\Z$}
A first reason why Jamison sets appear naturally in our study of \ka\ sets in $\Z$ is given by the following fact: 

\begin{fact}\label{Fac8}
 If $\qq$ is a \ka\ set in $\Z$, $\qq$ is a Jamison set.
\end{fact}

\begin{proof}
 Let $\varepsilon >0$ be a \ka\ constant for $\qq$. Let also $\lambda \in\T$ be such that $\sup_{\nq}|\lambda ^{n}-1|<\varepsilon $. Applying (2) of Theorem \ref{Theorem 1} to the measure 
$\sigma =\delta _{\{\lambda \}}$, we obtain that $\sigma (\{1 \})>0$, i.e.\ that $\lambda =1$. Hence $\qq$ is a Jamison set in $\Z$.
This argument shows that if $\varepsilon $ is a \ka\ constant for $Q$, then $\varepsilon^{2}/2$ is a Jamison constant for $Q$ (in the sense of the definition given after Theorem \ref{th0bis}). The following direct argument using operators shows that $\varepsilon$ is also a Jamison constant for $Q$. Indeed, if $\lambda\in\T$ is such that $\sup_{n\in Q}|\lambda^{n}-1|<\varepsilon$, the unitary operator $U$ given by multiplication by $\lambda$ on $\C$ satisfies $\sup_{n\in Q}||U^{n}\textbf{1}-\textbf{1}||<\varepsilon$. Hence $1$ is an eigenvalue of $U$, and $\lambda=1$.
\end{proof}

Fact \ref{Fac8} is actually a direct consequence of the observation that Jamison sets in $\Z$ are exactly \ka\ sets in $\Z$ for the class of irreducible unitary representations of $\Z$ (if $\mathcal{C}$ is a certain class of representations of a topological group $G$, a subset $Q$ of $G$ is a \ka\ set for $\mathcal{C}$ if there exists $\varepsilon>0$ such that property (ii) of Definition \ref{Definition 1.1} holds true for every representation $\pi$ of $G$ belonging to $\mathcal{C}$).
Not all Jamison sets give rise to \ka\ sets though. This can be seen 
thanks to the following example:

\begin{example}\label{Exa9}
 Let $(n_{k})_{\kk}$ be a sequence of positive integers such that $n_{0}=1$, and such that $n_{k}$ divides $n_{k+1}$ for every $k\ge 0$. Then $Q=\{n_{k}\,;\,\kk\}$ is not a \ka\ set in $\Z$. But (as shown in Example \ref{cor1}) if $\sup_{\kk}(n_{k+1}/n_{k})$ is finite, then $Q$ is a Jamison set in $\Z$.
\end{example}

\begin{proof}
 We will construct, for every $\varepsilon >0$, a continuous measure $\sigma$ on $\T$ such that $$\sup_{k\ge 0 }
|\wh{\sigma }(n_{k})-1|<\varepsilon .$$ The construction is essentially the same as the one given in \cite[Prop.\,3.9]{EG}. It uses infinite convolution of two-points Dirac measures, following an argument communicated by Jean-Pierre Kahane. Let $(a_{j})_{j\ge 1}$ be a decreasing sequence of positive real numbers with $a_{1}<\varepsilon/(2\pi) $ such that the series $\sum_{j\ge 1}a_{j}$ is divergent. Then the measure
\[
\sigma = \underset {j\geq 1}\Asterisk\bigl( (1-a_{j})\delta _{\{1\}}+a_{j}\delta _{\{e^{2i\pi /n_{j}}\}}\bigr)
\]
is well-defined on $\T$, and is a probability measure. Moreover, the assumption that the series $\sum_{j\ge 1}a_{j}$ diverges implies that the measure $\sigma $ is continuous. For every $k\ge 0$,
\[
\wh{\sigma }(n_{k})=\prod_{j\ge 1}\bigl( 1-a_{j}+a_{j}e^{2i\pi n_{k}/n_{j}}\bigr)
=\prod_{j\ge k+1}\bigl( 1-a_{j}(1-e^{2i\pi n_{k}/n_{j}})\bigr)
 \]
since $n_{j}$ divides $n_{k}$ for every $0\le j\le k$. As $| 1-a_{j}(1-e^{2i\pi n_{k}/n_{j}})|\le 1$, it follows that 
\[
|\wh{\sigma }(n_{k})-1|\le\sum_{j\ge k+1}a_{j}|1-e^{2i\pi n_{k}/n_{j}}|\le 2\pi \,a_{k+1}\,n_{k}\sum_{j\ge k+1}\dfrac{1}{n_{j}}\cdot 
\]
But since $n_{j}$ divides $n_{j+1}$ for all $j\ge 0$, $n_{j}\ge 2^{j-k}n_{k}$ for all $j\ge k+1$. Hence 
\[
|\wh{\sigma }(n_{k})-1|\le 2\pi \,a_{k+1}\,n_{k}\,\dfrac{1}{n_{k}}\sum_{j\ge 1}2^{-j}=2\pi a_{k+1}<\varepsilon \quad \textrm{ for every } k\ge 0
\]
by our assumption on the sequence $(a_{j})_{j\ge 1}$. Thus, by Theorem \ref{Theorem 2}, $\{n_{k}\,;\,\kk\}$ is not a \ka\ set in $\Z$.
\end{proof}

\subsection{Some examples of \ka\ and non-\ka\ sets in $\Z$}\label{Section 4}
We are now going to present a variety of examples of \ka\ and non-\ka\ sets in $\Z$. For this, we choose to use exclusively, and without further mention,  the characterizations of \ka\ and non-\ka\ sets provided by Theorem \ref{Theorem 1}
 and Corollary \ref{Cor7} respectively, although several of these examples could be obtained from the original definition of \ka\ sets as given in Definition \ref{Definition 1.1}.
Let us start our list of examples with the following obvious observation:

\begin{example}\label{Exa1}
 The set $\N$ is a \ka\ set in $\Z$, as well as any subset $\qq=\{n_{k}\,;\,\kk\}$ of $\Z$ such that 
$(n_{k}\theta )_{\kk}$ is uniformly distributed modulo $1$ for any $\theta\in\R\setminus\Z$.
\end{example}

\begin{proof}
Let $(n_k)_{k\ge 0}$ be such that $(n_{k}\theta )_{\kk}$ is uniformly distributed modulo $1$ for any $\theta\in\R\setminus\Z$, i.e. that
$$\frac{1}{N}\sum_{k=1}^{N}\lambda^{n_{k}}\!\xymatrix@C=17pt{\ar[r]&}\!0 \quad \textrm{as }N \!\xymatrix@C=17pt{\ar[r]&}\!+\infty$$
for every $\lambda\in\T\setminus\{1\}$.
 Let $\sigma $ be a probability measure on $\T$ such that $\sup_{k\ge 1}|\wh{\sigma }(n_{k})-1|<1$. Since
\[
\dfrac{1}{N}\sum_{k=1}^{N}\wh{\sigma }(n_{k})=\int_{\T}\Bigl( \dfrac{1}{N}\sum_{k=1}^{N}\lambda ^{n_{k}}\Bigr)d\sigma (\lambda )\!\xymatrix@C=17pt{\ar[r]&}\!\sigma (\{1\}) \;\; \textrm{ as } N\!\xymatrix@C=17pt{\ar[r]&}\!+\infty,
\]
we have $\sigma(\{1\})>0$.
\end{proof}

On the other hand, the sets $a\N$, where $a\ge 2$, are obviously never \ka\ sets in $\Z$ (they do not generate $\Z$).

\begin{example}\label{Exa2}
Let $a\ge 2$ be an integer. The set $a\Z$, as well as all infinite subsets of $a\Z$, are never \ka\ sets in $\Z$.
\end{example}

We next provide an example showing that $\{n_{k}\,;\,\kk\}$ may be a \ka\ set in $\Z$ without $(n_{k}\theta )_{\kk}$ being uniformly distributed modulo $1$ for every $\theta \in\R\setminus\Q$.
\begin{example}\label{Exa3}
 The set $\{2^{k}+k\,;\,k\ge 0\}$ is a \ka\ set in $\Z$, although there exists an irrational number $\theta$ such that the sequence $((2^{k}+k)\theta )_{\kk}$ is not dense in $[0,1]$ modulo $1$, hence not uniformly distributed modulo $1$.
\end{example}

 \begin{proof}
  The sequence $(n_{k})_{\kk}$ defined by $n_{k}=2^{k}+k$, $\kk$, satisfies the relation $2n_{k}=n_{k+1}+k-1$. Let $\sigma $ be a probability measure on $\T$ such that $\sup_{k\ge 0}|\wh{\sigma }(n_{k})-1|<1/18$. Since, by the Cauchy-Schwarz inequality,
  \[
  |\wh{\sigma }(k)-1|\le\int_{\T}|\lambda^{k}-1|d\sigma(\lambda)\le\sqrt{2}\,|\wh{\sigma }(k)-1|^{1/2}\quad\textrm{ for every } k\in\Z,
  \]
  we have
 \begin{align*}
|\wh{\sigma }(k-1)-1|&\le 2 \int_{\T}|\lambda^{n_{k}}-1|d\sigma(\lambda)
+\int_{\T}|\lambda^{n_{k+1}}-1|d\sigma(\lambda)\\
&\le 2\sqrt{2}\,
|\wh{\sigma }(n_{k})-1|^{1/2}+\sqrt{2}\,|\wh{\sigma }(n_{k+1})-1|^{1/2}
\end{align*}
for all $k\ge 1$,
so
that $\sup_{k\ge 0}|\wh{\sigma }(k)-1|<1$. Hence $\sigma( \{1\})>0$ by the proof of Example \ref{Exa1}. So $\{n_{k}\,;\,\kk\}$ is a \ka\ set in $\Z$. But $(n_{k})_{k\ge 0}$ being lacunary, it follows from a result due independently to Pollington \cite{Pol} and De Mathan \cite{DM} that there exists a subset $A$ of $[0,1]$ of Hausdorff measure $1$ such that for every $\theta $ in $ A$, the set
 $\{n_{k}\theta \textrm{ ; } k\ge 0\}$ is not dense modulo $1$. One of these numbers $\theta $ is irrational, and $((2^{k}+k)\theta )_{k\ge 0}$ is of course not uniformly distributed modulo $1$ for this particular choice of $\theta $. 
 \end{proof}

The same kind of argument allows us to show the following result (which also follows from Theorem \ref{The12}):

\begin{example}\label{Exa4}
 The set $\mathcal{P}$ of prime numbers is a \ka\ set in $\Z$.
\end{example}

\begin{proof}
A result of Vinogradov \cite{V} states that any sufficiently large odd integer
can be written as a sum of at most three prime numbers; therefore every sufficiently
large integer can be written as a sum of at most four prime numbers.
It follows from the same reasoning as in Example \ref{Exa3} above that if $\sigma $ is a probability measure on $\T$ such that 
$\sup_{p\in\mathcal{P}}|\wh{\sigma }(p)-1|<1/32$, there exists an integer $k_{0}$ such that $\sup_{k\ge k_{0}}|\wh{\sigma }(k)-1|<1$. The same argument as in Example \ref{Exa1} then implies that $\sigma (\{1\})>0$ and thus $\mathcal{P}$ is a Kazhdan set in $\Z$.
 
The above proof combined with Remark \ref{rmk:47} implies that $(\mathcal{P},1/32)$ is a Kazhdan pair in $\Z$. A better Kazhdan constant for $\mathcal{P}$, namely $\sqrt{2}/4$, can be obtained using Vinogradov's result and the following argument. Assume that $U$ is a unitary operator
acting on a Hilbert space $H$ and that $x\in H$ is a norm one vector such that $\sup_{p\in\mathcal{P}}\|U^px - x\| < \sqrt{2}/4$. If $p_j$, $1\le j \le 4$, are four prime numbers, then 
$$\|U^{p_1+p_2+p_3+p_4}x - x\| \le  \sum_{j=1}^4 \|U^{p_j}x - x\| < \sqrt{2} .$$
As every sufficiently
large integer can be written as a sum of at most four prime numbers, we obtain the existence of an integer $k_{0}$ such that $\sup_{k\ge k_{0}}\|U^kx-x\|< \sqrt{2}$. The classical argument showing that $(G,\sqrt{2})$ is a Kazhdan pair for any topological group $G$ (see \cite[Prop. 1.1.5]{BdHV}) applies in our situation. We obtain that $1$ is an eigenvalue of $U$, and it thus follows that $(\mathcal{P},\sqrt{2}/4)$ is a Kazhdan pair in $\Z$.   
\end{proof}

More generally:

\begin{proposition}\label{Pro5}
 Let $(n_{k})_{k\ge 0}$ be a sequence of elements of $\Z$. Suppose that there exist an integer $a\ge 2$ and integers
$b_{0},\dots,b_{p}\in\Z$, not all zero, such that 
\begin{enumerate}
 \item [\emph{(i)}] the set $\{b_{0}n_{k}+b_{1}n_{k+1}+\cdots+b_{p}n_{k+p}\,;\,k\ge 0\}$ contains all sufficiently large multiples of $a$;
\item[\emph{(ii)}] there exists $k\ge 0$ such that $n_{k}$ and $a$ have no non-trivial common divisor.
\end{enumerate}
Then $\{n_{k}\,;\,\kk\}$ is a \ka\ set in $\Z$.
\end{proposition}

\begin{proof}
 We can suppose without loss of generality that $n_{0}$ and $a$ have no non-trivial common divisor. Let $1\le r\le a-1$ be such that $n_{0}\equiv r\,\textrm{mod} \, a$, and let $C$ denote the set of all the $a$-th roots of $1$ which are different from $1$. Then $\lambda ^{r}\neq 1$ for every $\lambda \in C$ (if $\lambda^{a}=\lambda^{r}=1$ and $\lambda\not =1$, $a$ and $r$ have a non-trivial common divisor), and hence there exists $0<\delta <1/2$ such that $1-\Re e(\lambda ^{r})>12\delta (a-1)$ for every $\lambda \in C$. Let $\sigma $ be a probability measure on $\T$ such that 
 $$
 \sup_{k\ge 0}|\wh{\sigma }(n_{k})-1|<\delta^{2} \,\dfrac{1}{2(|b_{0}|+|b_{1}|+\cdots+|b_{p}|)^{2}}.
 $$
Using again the chain of inequalities 
\[
  |\wh{\sigma }(j)-1|\le\int_{\T}|\lambda^{j}-1|d\sigma(\lambda)\le\sqrt{2}\,|\wh{\sigma }(j)-1|^{1/2},\quad j\in\Z,
  \]
we obtain that
$$
\sup_{k\ge 0}|\wh{\sigma }(b_{0}n_{k}+\cdots+b_{p}n_{k+p})-1|<\delta.
$$
Suppose now that $\sigma( \{1\})=0$. Assumption (i) and an argument similar to the one employed in the proof of Example \ref{Exa1} imply that $|\sigma (C)-1|\le\delta $, so that $\sigma (C)\ge 1-\delta $. Let $\tau $ be the probability measure on $\T$ supported on $C$ defined by $\tau =\sigma (C)^{-1}\textbf{1}_{C}\sigma $. We write $\tau$ as $\tau =\sum_{\lambda \in C}a_{\lambda }\delta _{\{\lambda \}}$, where $a_{\lambda }\ge 0$ for each $\lambda \in C$ and $\sum_{\lambda \in C}a_{\lambda }=1$. We have 
\[
 \wh{\tau }(n_{0})=\sum_{\lambda \in C}a_{\lambda }\lambda ^{n_{0}}=\sum_{\lambda \in C}a_{\lambda }\lambda ^{r}
\]
and
 \[
  \wh{\tau }(n_{0})=\wh{\sigma }(n_{0})+\left(\frac{1}{\sigma(C)}-1\right)\wh{\sigma }(n_{0})-\frac{1}{\sigma(C)}\int_{\T\setminus C}\lambda^{n_{0}}d\sigma(\lambda).
 \]
Hence 
 \[
|\wh{\tau }(n_{0})-1|\le|\wh{\sigma }(n_{0})-1|+2\,\frac{1-\sigma(C)}{\sigma(C)}
\le|\wh{\sigma }(n_{0})-1|+\frac{2\delta}{1-\delta }\le \frac{3\delta}{1-\delta }<6\delta
\]
since $\delta<1/2$.
This implies that
$a_{\lambda }(1-\Re e(\lambda ^{r}))<6\delta $ for every $\lambda \in C$. Since $\sum_{\lambda \in C}a_{\lambda }=1$ and $C$ has cardinality $a-1$, there exists $\lambda _{0}\in C$ such that $a_{\lambda _{0}}>(1-\delta )/(a-1)$. We thus obtain that $1-\Re e (\lambda _{0}^{r})<6\delta (a-1)/(1-\delta) <12\delta (a-1)$, which contradicts our assumption on $\delta $. Hence $\sigma (\{1\})>0$, and $\{n_{k}\,;\,\kk\}$ is a \ka\ set in $\Z$.
\end{proof}

As a consequence of Proposition \ref{Pro5}, we obtain for instance:

\begin{example}\label{Exa6}
 The set of squares $\{k^{2}\,;\,k\ge 1\}$ is a \ka\ set in $\Z$.
\end{example}

\begin{proof}
 It suffices to observe that $(k+2)^{2}-k^{2}=4(k+1)$, while there are some square numbers which are not divisible by $4$.
\end{proof}

This last result will be generalized in Section \ref{Section 7} as a consequence of Theorem \ref{The12}, see Example \ref{Exa15} below. 

Examples \ref{Exa4} and \ref{Exa6} above can also be derived from the following observation. A subset $Q$ of $\N$ is said to have positive \emph{Shnirel'man density} if $\inf_{N\ge 1} \frac{1}{N}\#[1,N]\cap Q$ is positive. Any set $Q$ with positive Shnirel'man density has the property that there exists an integer $d\ge 1$ such that every element of $\N$ can be written as a sum of at most $d$ elements of $Q$ (not necessarily distinct). See \cite[Ch. 7]{Nath} for more on this topic. We deduce from this observation that

\begin{example}
 Any subset $Q$ of $\N$ with positive Shnirel'man density is a \ka\ set in $\Z$.
\end{example}

The set $\{0,1\}\cup(\mathcal{P}+
\mathcal{P})=\{0,1\}\cup\{p+q\textrm{ ; }p,q\in\mathcal{P}\}$  has positive Shnirel'man density (see \cite[Th. 7.8]{Nath}), and this implies again (see also \cite[Th. 7.9]{Nath}) that $\mathcal{P}$ is a \ka\ set in $\Z$. 
\par\smallskip
We finish this section with a last example along these lines. It immediately follows from \cite[Th. 7.10]{Nath}, which states that if $Q$ is a set of primes containing a positive proportion of the primes, every sufficiently large element of $\N$ can be written as the sum of a bounded numbers of elements of $Q$.

\begin{example}
 Let $Q$ be a set of primes that contains a positive proportion of the
primes, that is, there exists a constant $\theta\in (0,1)$ such that $ \#([1,N]\cap Q) > \theta
\,\#([1,N]\cap \mathcal{P})$ for
all  $N$ sufficiently large. Then $Q$ is a Kazhdan set in $\Z$.
\end{example}

\section{Proof of Theorem \ref{Theorem 2}}\label{Section 6}
 The implication $(1)\Longrightarrow (2)$ being obvious, we prove that $(2)$ implies $(1)$. Reasoning by contradiction, we suppose that $(2)$ holds true, but that $\qq$ is not a \ka\ set in $\Z$. We have to consider separately two cases, depending on whether $\qq$ is a Jamison set in $\Z$ or not. 
\par\smallskip 
\textbf{Case 1.} If $\qq$ is not a Jamison set in $\Z$, Theorem \ref{The0Ter} implies that there exists for every $\varepsilon >0$ a perfect compact subset $K_{\varepsilon }$ of $\T$ such that $\sup_{\nq}|\lambda ^{n}-1|<\varepsilon $ for every $\lambda \in K_{\varepsilon}$. Any continuous probability measure $\sigma $ on $\T$ which is supported on $K_{\varepsilon }$ satisfies $\sup_{\nq}|\wh{\sigma }(n)-1|\le\varepsilon $, and this violates our supposition that $(2)$ of Theorem \ref{Theorem 2} holds true.
\par\medskip 
\textbf{Case 2.} Suppose now that $\qq$ is a Jamison set in $\Z$, and let $\varepsilon_{0} $ be a Jamison constant for $\qq$ (Jamison constants are defined after the proof of Theorem \ref{th0bis}). The first step of the proof in this case is to show the following lemma:

\begin{lemma}\label{Lem12}
 Under the assumptions above, there exists for every $\varepsilon >0$ a discrete probability measure $\sigma $ on $\T$ which has the following properties:
\begin{enumerate}
 \item [\emph{(a)}] $\sup_{\nq}|\wh{\sigma }(n)-1|<\varepsilon $;
\item[\emph{(b)}] $\sigma $ is supported on a finite subset $F$ of $\T\setminus \{1\}$: $\sigma =\sum_{\lambda \in F}a_{\lambda }\delta _{\lambda }$, with $a_{\lambda }>0$
for every $\lambda \in F$ and $\sum_{\lambda \in F}a_{\lambda }=1$;
\item[\emph{(c)}] $0<a_{\lambda }<\varepsilon /\varepsilon _{0}^{2}$ for every $\lambda \in F$;
\item[\emph{(d)}] $F$ is contained in the arc $\Gamma _{\varepsilon }=\{\lambda \in \T\,;\, |\lambda -1|<\varepsilon \}$.
\end{enumerate}
\end{lemma}

\begin{proof}[Proof of Lemma \ref{Lem12}] Since $\qq$ is not a \ka\ set in $\Z$, there exists for each $p\ge 1$ a probability measure $\mu _{p}$ on $\T$ with $\mu _{p}(\{1\})=0$ such that $\sup_{\nq}|\wh{\mu }_{p }(n)-1|<2^{-p}$. We denote by $\mu _{p,d}$ and $\mu _{p,c}$ respectively the discrete  and continuous parts of the measure $\mu _{p}$. First, we have:

\begin{claim}\label{claim}
 We have $\underline{\lim}_{\,p\to+\infty\, }\mu _{p,d}(\T)=\gamma >0$.
\end{claim}

\begin{proof}[Proof of Claim \ref{claim}]
Suppose that  $\underline{\lim}_{\,p\to+\infty\, }\mu _{p,d}(\T)=0$. By passing to a subsequence, we can assume that $\mu _{p,c}(\T)$ tends to $1$ as $p$ tends to infinity, and we denote for each $p\ge 1$ by $\nu _{p}$ the measure $\nu _{p}=\mu _{p,c}/\mu _{p,c}(\T)$; it is a continuous  probability measure on $\T$. For each $n\in\Z$, we have
\begin{align*}
  \wh{\nu} _{p}(n)-1&= \dfrac{1}{\mu _{p,c}(\T)}\bigl( \wh{\mu }_{p,c}(n)-1\bigr)+\dfrac{1}{\mu _{p,c}(\T)}-1\,\\
\intertext{so that}
|\wh{\nu} _{p}(n)-1|&\le \dfrac{1}{\mu _{p,c}(\T)}|1- \wh{\mu }_{p,c}(n)|+\dfrac{1}{\mu _{p,c}(\T)}-1\,\\
&\le \dfrac{1}{\mu _{p,c}(\T)}\left(2^{-p}+\mu_{p,d}(\T)\right)+\dfrac{1}{\mu _{p,c}(\T)}-1\,.\\
\end{align*}
Since the right-hand bound in this expression tends to zero as $p$ tends to infinity, this yields that there exists for every $\delta >0$ a continuous probability measure $\nu $ on $\T$ such that $\sup_{\nq}|\wh{\nu  }(n)-1|<\delta $, which contradicts assumption $(2)$ of Theorem \ref{Theorem 2}. So $\underline{\lim}_{\,p\to+\infty \,}\mu _{p,d}(\T)=\gamma >0$, as claimed. 
\end{proof}

Now, observe that for every $n\in\Z$ and $p\ge 1$,
$$\Re e \,\wh{\mu }_{p}(n)=\Re e\, \wh{\mu }_{p,d}(n)+\Re e\, \wh{\mu }_{p,c}(n)\le \Re e \,\wh{\mu }_{p,d}(n)+\mu_{p,c}(\T)\le \Re e \,\wh{\mu }_{p,d}(n)+1-\mu_{p,d}(\T),$$
so that
$\mu _{p,d}(\T)-\Re e\,\wh{\mu }_{p,d}(n)\le 1-\Re e\,\wh{\mu }_{p}(n)$. Hence
\begin{align*}
1-\dfrac{1}{\mu_{p,d}(\T)}\,\Re e\,\wh{\mu }_{p,d}(n)&\le \dfrac{1}{\mu _{p,d}(\T)}\bigl(1-\Re e\, \wh{\mu }_{p}(n)\bigr).
\intertext{It follows that the discrete probability measures $\tau _{p}=\mu _{p,d}/\mu _{p,d}(\T)$, $p\ge 1$, satisfy}
\sup_{\nq}|1-\wh{\tau }_{p}(n)|^{2}&\le\dfrac{1}{\mu _{p,d}(\T)}\,2^{-p+1}.
\end{align*}
Combining these inequalities with
Claim \ref{claim}, we obtain  that there exists for every $\delta >0$ a discrete probability measure $\tau $ on $\T$ with $\tau (\{1\})=0$ such that 
$\sup_{\nq}|1-\wh{\tau }(n)|<\delta $. Without loss of generality, we can suppose that the measure $\tau $ is supported on a finite set. Write $\tau$ as $$\tau =\sum_{\lambda \in G}b_{\lambda }\delta _{\{\lambda\} },$$ where $G$ is a finite subset of $\T\setminus \{1\}$, $b_{\lambda }>0$ for every $\lambda \in G$, and $\sum_{\lambda \in G}b_{\lambda }=1$. We have for every $\lambda \in G$ and $\iz$
\[
0\le b_{\lambda }\bigl( 1-\Re e\,\lambda ^{n}\bigr)\le 1-\Re e\,\wh\tau 
(n)\le|1-\wh{\tau }(n)|
\]
so that $\sup_{\nq}(1-\Re e\,\lambda ^{n})<\delta /b_{\lambda }$. It follows that $\sup_{\nq}|1-\lambda ^{n}|^{2}<(2\delta) /b_{\lambda }$, i.e.\ that 
$\sup_{\nq}|1-\lambda ^{n}|<\sqrt{(2\delta) /b_{\lambda }}$. Now, recall that $\lambda \neq 1$, and that $\qq$ is supposed to be a Jamison set with Jamison constant $\varepsilon _{0}$. This implies that $\varepsilon _{0}<\sqrt{(2\delta) /b_{\lambda }}$, i.e.\ that $$0<b_{\lambda }<\frac{2\delta }{\varepsilon _{0}^2}\quad\textrm{ for every }\lambda \in G.$$
\par\smallskip
The assumption that $1$ belongs to $\qq$ now comes into play. It implies that $|\wh{\tau }(1)-1|<\delta $, so that 
\[
\int_{\T}|1-\lambda |^{2}d\tau (\lambda )=2\int_{\T}(1-\Re e(\lambda ))\,d\tau (\lambda )<2\delta .
\]
By the Markov inequality, $\sqrt{\delta }\,\tau (\T\setminus \Gamma _{\delta^{1/4}})<2\delta $, where $\Gamma _{\delta^{1/4} }$ denotes the arc $$\Gamma _{\delta^{1/4} }=\{\lambda \in \T\,;\,|\lambda -1|<\delta ^{1/4}\}.$$ Thus $\tau (\Gamma _{\delta^{1/4}})>1-2\sqrt{\delta }$. Let now $\sigma $ be the probability measure on $\T$ defined by $$\sigma =\tau \textbf{1}_{\Gamma_{\delta ^{1/4}} }/\tau (\Gamma _{\delta^{1/4} }).$$ It is a discrete probability measure, supported on the set $G\cap \Gamma _{\delta^{1/4} }$, which has the form
$$\sigma =\sum_{\lambda \in G\cap \Gamma _{\delta ^{1/4}}}a_{\lambda }\delta _{\{\lambda\}},$$ with $a_{\lambda }=b_{\lambda }/\tau (\Gamma _{\delta ^{1/4}})$ for every $\lambda \in G\cap \Gamma _{\delta^{1/4} }$. Hence
\begin{align*}
 0<a_{\lambda }&<\dfrac{\delta }{1-2\sqrt{\delta }}\cdot\dfrac{2}{\varepsilon _{0}^{2}}\quad
\textrm{for every } \lambda \in G\cap \Gamma _{\delta^{1/4} }.
\end{align*}
Moreover, we have for every $n\in\Z$
\begin{align*}
|\wh{\sigma }(n)-1|&\le\dfrac{1}{\tau (\Gamma _{\delta^{1/4} })}\,\Bigl|\wh{\tau \textbf{1}}_{\Gamma _{\delta ^{1/4}}}(n)-1\Bigr|+\dfrac{1}{\tau (\Gamma _{\delta ^{1/4}})}-1\\
&\le\dfrac{1}{\tau (\Gamma _{\delta ^{1/4}})}\,|\wh{\tau }(n)-1|+\dfrac{1}{\tau (\Gamma _{\delta^{1/4} })}\,\tau (\T\setminus \Gamma _{\delta^{1/4} })+\dfrac{1}{\tau (\Gamma _{\delta^{ 1/4}})}-1
\end{align*}
{so that}
\[
\sup_{\nq}|\wh{\sigma }(n)-1|<\dfrac{\delta+ 4\sqrt{\delta }}{1-2\sqrt{\delta }}\cdot 
\]
Let now $\varepsilon >0$. If $\delta >0$ is so small that $(\delta +4\sqrt{\delta })/(1-2\sqrt{\delta })<\varepsilon$, $(2\delta)/(1-2\sqrt{\delta})<\varepsilon$ and $\delta ^{1/4}<\varepsilon $, the associated measure $\sigma $ satisfies properties (a), (b), (c), and (d) of Lemma \ref{Lem12}.
\end{proof}

Our aim is now to use Lemma \ref{Lem12} in order to construct, for every $\varepsilon >0$, a continuous probability measure $\sigma $ on $\T$ such that $\sup_{\nq}|\wh{\sigma }(n)-1|<\varepsilon $. This will contradict assumption (2), and conclude the proof of Theorem \ref{Theorem 2}.
\par\smallskip 
Let us fix $\varepsilon >0$, and let $(\varepsilon _{p})_{p\ge 1}$ be a sequence of positive real numbers decreasing very fast to $0$ (how fast will be specified in the sequel of the proof). Let $(\sigma _{p})_{p\ge 1}$ be a sequence of probability measures associated to $(\varepsilon _{p})_{p\ge 1}$ given by Lemma \ref{Lem12}. We write each measure $\sigma _{p}$ as $$\sigma _{p}=\sum_{\lambda\in F_{p}}a_{\lambda ,p}\,\delta _{\{\lambda \}},$$ where $F_{p}$ is a finite subset of $\T\setminus \{1\}$ contained in $\Gamma _{\varepsilon _{p}}$, $0<a_{\lambda ,p}<\varepsilon _{p}/\varepsilon _{0}^{2}$ for every $\lambda \in F_{p}$, and $\sum_{\lambda \in F_{p}}a_{\lambda ,p}=1$. We have $\sup_{\nq}|\wh{\sigma}_{p} (n)-1|<\varepsilon _{p}$ for every $p\ge 1$.
\par\smallskip 
For each $p\ge 1$, let us denote by $F_{1}\cdot F_{2}\, \cdots\, F_{p}$ the set $$F_{1}\cdot F_{2}\, \cdots\, F_{p}=\{\lambda _{1}\lambda _{2}\dots\lambda _{p}\,;\,\lambda _{i}\in F_{i}\ \textrm{for each}\ i=1,\dots,p\}.$$ We construct the numbers $\varepsilon _{p}$ by induction on $p\ge 1$ in such a way that they satisfy the following condition:
\[
\inf\bigl\{|\lambda -\lambda '|\,;\,\lambda ,\,\lambda '\in F_{1}\cdot F_{2}\, \cdots\, F_{p-1},\ \lambda \neq\lambda '\bigr\}>4\,\varepsilon _{p}.\]

\begin{claim}\label{claim2}
 Under this assumption, we have $\# F_{1}\cdot F_{2}\, \cdots\, F_{p-1}=\#F_{1}\times\dots\times\#F_{p}$ for every $p\ge 1$. In other words, the numbers $\lambda _{1}\dots\lambda _{p}$, $\lambda _{i}\in F_{i}$ for each $i=1,\dots,p$, are all distinct.
\end{claim}

\begin{proof}[Proof of Claim \ref{claim2}]
 The proof is a simple induction on $p$. Suppose that the assumption is true for some $p\ge 1$. If $\lambda _{1}\dots\lambda _{p}$ and $\lambda _{1}'\dots\lambda '_{p}$ are such that $\lambda _{i},\,\lambda '_{i}\in F_{i}$ for each $i=1,\dots,p$ and $\lambda _{1}\dots\lambda _{p}=\lambda '_{1}\dots\lambda _{p}'$, we need to show that $\lambda _{i}=\lambda _{i}'$ for every $i=1,\dots,p$. If $\lambda _{p}=\lambda '_{p}$, then $\lambda _{1}\dots\lambda _{p-1}=\lambda _{1}'\dots\lambda _{p-1}'$, and our induction assumption implies that $\lambda _{i}=\lambda _{i}'$ for every $i=1,\dots,p-1$. So let us suppose that $\lambda _{p}\neq\lambda '_{p}$. Writing $\lambda =\lambda _{1}\dots\lambda _{p-1}$ and $\lambda '=\lambda _{1}'\dots\lambda '_{p-1}$, we have $\lambda \neq\lambda '$ and thus
\[
|\lambda \lambda _{p}-\lambda '\lambda '_{p}|\ge|\lambda -\lambda '|-|\lambda _{p}-\lambda '_{p}|>4\varepsilon _{p}-2\varepsilon _{p}=2\varepsilon _{p}
\]
since $\lambda _{p},\,\lambda '_{p}\in \Gamma _{\varepsilon _{p}}$ and $\lambda ,\lambda '\in F_{1}\cdot F_{2}\, \cdots\, F_{p-1}$ with $\lambda \neq\lambda '$. So $\lambda \lambda _{p}\neq\lambda '\lambda '_{p}$, which contradicts our initial assumption. Hence $\# F_{1}\cdot F_{2} \, \cdots\, F_{p-1}=\#F_{1}\times\dots\times\#F_{p}$.
\end{proof}

We now additionally require that the sequence $(\varepsilon _{p})_{p\ge 1}$ decreases to zero in such a way that the series $\sum_{p\ge 1}\varepsilon _{p}$ is convergent. This implies that we can define a probability measure $\sigma $ as the infinite convolution product of the measures $\sigma _{p}$: $$\sigma =\underset {p\geq 1}\Asterisk \sigma _{p}.$$ In order to prove that the measure $\sigma$ is well-defined, it suffices to check that the infinite product $\prod_{p\ge 1}\wh{\sigma}_{p}(n)$ converges for every $n\in \Z$, and this is straightforward using the facts that 
\[
|\wh{\sigma}_{p} (n)-1|\le\int_{F_{p}}|\lambda ^{n}-1|\,d\sigma _{p}(\lambda )\le n\sup_{\lambda \in F_{p}}|\lambda -1|<n\varepsilon _{p}
\]
for every $p\ge 1$ and every $n\in\Z$ and that the series $\sum_{p\ge 1}\varepsilon _{p}$ is convergent. For every $n\in\Z$, we have
\begin{align*}
 \wh{\sigma }(n)&=\prod_{p\ge 1}\wh{\sigma }_{p}(n)\\[-2ex]
\end{align*}
and
\begin{align*}
\sup_{\nq}|\wh{\sigma }(n)-1|&\le \sup_{\nq}\sum_{p\ge 1}|\wh{\sigma}_{p}(n)-1|
\le\sum_{p\ge 1}\sup_{\nq}|\wh{\sigma}_{p}(n)-1|\le\sum_{p\ge 1}\varepsilon _{p}.
\end{align*}
If we suppose that the sequence $(\varepsilon _{p})_{p\ge 1}$ is such that $\sum_{p\ge 1}\varepsilon _{p}<\varepsilon $, then the measure $\sigma $ satisfies $\sup_{\nq}|\wh{\sigma }(n)-1|<\varepsilon $.  In order to reach the contradiction we are looking for, it remains to prove that:

\begin{claim}\label{claim3}
 The measure $\sigma  $ is continuous.
\end{claim}

 Claim \ref{claim3} follows from rather standard arguments which rely on the fact that $\sigma $ is by construction a Cantor-like measure.
\begin{proof}[Proof of Claim \ref{claim3}]
For every $p\ge 1$, let $\tau _{p}$ be the partial convolution product 
\[
\tau _{p}=\sigma _{1}*\sigma _{2}*\,\dots\,*\sigma _{p}=\sum_{\lambda _{1}\in F_{1},\dots,\lambda _{p}\in F_{p}}a_{\lambda _{1},1}a_{\lambda _{2},2}\dots a_{\lambda _{p},p}\,\delta _{\{\lambda _{1}\dots\lambda _{p}\}}.
\]
The support of $\tau _{p}$ consists of the $N_{p}=\#F_{1}\times\#F_{2}\times\cdots\times\#F_{p}$ distinct points of the set $F_{1}\cdot F_{2}\, \cdots\, F_{p}$. Whenever $\lambda $ and $\lambda '$ are two distinct elements of this set, $|\lambda -\lambda '|>4\varepsilon _{p+1}$. For each $\lambda \in F_{1}\cdot F_{2}\, \cdots\, F_{p}$, let $\Gamma _{\lambda ,p}$ denote the arc
\[
\Gamma _{\lambda ,p}=\bigl\{\mu \in \T\,;\,|\mu -\lambda |<2\varepsilon _{p+1}\bigr\}\cdot\ 
\]
All the arcs $\Gamma _{\lambda ,p}$, $\lambda \in F_{1}\cdot F_{2}\, \cdots\, F_{p}$, are pairwise disjoint. We denote by $K_{p}$ the following compact subset of $\T$:
\[
K_{p}=\bigcup_{\lambda \in F_{1}\cdot F_{2}\, \cdots\, F_{p}}\Gamma _{\lambda ,p}.
\]
If $\lambda \in F_{1}\cdot F_{2}\, \cdots\, F_{p}$ and $\lambda _{p+1}\in F_{p+1}$, $|\lambda \lambda _{p+1}-\lambda |<\varepsilon _{p+1}$, and if $2\varepsilon _{p+2}<\varepsilon _{p+1}$, it follows that the arc $\Gamma _{\lambda \lambda _{p+1},p+1}$ is contained in $\Gamma _{\lambda ,p}$. The sequence $(K_{p})_{p\ge 1}$ is hence decreasing, and $\tau_{q}(K_{p})=1$ for every $q\ge p$.
\par\smallskip 
We now aim to show that for every $q\ge p$ and every $\lambda \in F_{1}\cdot F_{2}\, \cdots\, F_{p}$, which we write as $\lambda =\lambda _{1}\dots\lambda _{p}$ with $\lambda _{i}\in F_{i}$ for every $i=1,\dots,p$, we have $\tau _{q}(\Gamma _{\lambda ,p})=a_{\lambda _{1},1}\dots a_{\lambda _{p},p}$. For this it suffices to prove that for every $\lambda' \in F_{1}\cdot F_{2}\, \cdots\, F_{p}$ and every $\lambda _{i}\in F_{i}$, $i=p+1,\dots,q$, the point $\lambda '\lambda _{p+1}\dots\lambda _{q}$ belongs to $\Gamma _{\lambda ,p}$ if and only if $\lambda =\lambda '$.
Indeed, if this last statement is true we will have 
\begin{align*}
 \tau_{q}(\Gamma _{\lambda ,p})&=
\sum_{\lambda _{p+1}\in F_{p+1},\dots,\lambda _{q}\in F_{q}}
a_{\lambda _{1},1}\dots a_{\lambda _{q},q}\\
&=a_{\lambda _{1},1}\dots a_{\lambda _{p},p}\left(\sum_{\lambda _{p+1}\in F_{p+1},\dots,\lambda _{q}\in F_{q}}
a_{\lambda _{p+1},p+1}\dots a_{\lambda _{q},q}\right)\\
&=a_{\lambda _{1},1}\dots a_{\lambda _{p},p}\prod_{i=p+1}^{q}\left(\sum_{\lambda_{i}\in F_{i}}a_{\lambda_{i},i}\right)=a_{\lambda _{1},1}\dots a_{\lambda _{p},p}.
\end{align*}
With the notation above, observe first that 
\[
|\lambda \lambda _{p+1}\dots\lambda _{q}-\lambda |\le\sum_{i=p+1}^{q}|\lambda _{i}-1|<\sum_{i\ge p+1}\varepsilon _{i}.
\]
If the sequence $(\varepsilon _{p})_{p\ge 1}$ decreases so fast that $\sum_{i\ge p}\varepsilon _{i}<2\varepsilon _{p}$ for every $p\ge 1$, the right-hand bound in the display above is less than $2\varepsilon _{p+1}$, so that $\lambda \lambda _{p+1}\dots\lambda _{q}$ indeed belongs to the arc $\Gamma _{\lambda ,p}$.
Conversely, suppose that $\lambda \neq\lambda '$. Then $\lambda '\lambda _{p+1}\dots\lambda _{q}$ belongs to the arc $\Gamma _{\lambda ',p}$ which is disjoint from $\Gamma _{\lambda ,p}$. This proves our claim that $\tau _{q}(\Gamma _{\lambda ,p})=a_{\lambda _{1},1}\dots a_{\lambda _{p},p}$ for every $q\ge p$ and $\lambda =\lambda _{1}\dots \lambda _{p}$ with $\lambda _{i}\in F_{i} $ for each $i=1,\dots,p$.
\par\smallskip 
Observe now that if we endow the set $\mathcal{P}(\T)$ of probability measures on $\T$ with the topology of $w^{*}$-convergence of measures, the map $m\mapsto m(F)$ is upper semi-continuous on $\mathcal{P}(\T)$ for every closed subset $F$ of $\T$.
The arcs $\Gamma _{\lambda ,p}$ being closed and disjoint, and the measure $\sigma$ being the $w^{*}$-limit of the sequence of measures $(\tau_{q})_{q\ge 1}$, it follows that $\sigma (\Gamma _{\lambda ,p})=a_{\lambda _{1},1}\dots a_{\lambda _{p},p}$ for every $\lambda =\lambda _{1}\dots\lambda _{p}$ with $\lambda _{i}\in F_{i}$ for every $i=1,\dots,p$, and also that $\sigma  (K_{p})=1$. The measure $\sigma $ is thus supported on the compact set $K=\bigcap_{p\ge 1}K_{p}$, and there exists for every $\mu \in K$ a unique sequence $(\lambda _{i})_{i\ge 1}$ with $\lambda _{i}\in F_{i}$ for each $i\ge 1$ such that $\mu $ belongs to $\Gamma _{\lambda _{1}\dots\lambda _{p},p}$ for every $p\ge 1$. Thus $$0\le \sigma (\{\mu \})\le\sigma (\Gamma _{\lambda _{1}\dots\lambda _{p},p})=a_{\lambda _{1},1}\dots a_{\lambda _{p},p}\quad \textrm{ for every }p\ge 1.$$ It is at this point that we use our assumption that $0<a_{\lambda _{i},i}<\varepsilon _{i}/\varepsilon _{0}^{2}$ for every $i\ge 1$: it yields that
\[ 0\le
\sigma (\{\mu \})\le\prod_{i=1}^{p}\ \dfrac{\varepsilon _{i}}{\varepsilon _{0}^{2}\vphantom{ \textrm{\Large A}}}\quad\textrm{for every}\ p\ge 1.
\]
If $\varepsilon _{0}^{-2p}\varepsilon _{1}\dots\varepsilon _{p}$ tends to zero as $p$ tends to infinity, which can always be assumed provided $\varepsilon _{p}$ tends to zero sufficiently quickly, then $\sigma (\{\mu \})=0$. This being true for every $\mu \in K$, the measure $\sigma $ is continuous.
\end{proof}
We have thus constructed, for every $\varepsilon>0$, a continuous probability measure $\sigma$ on 
$\T$ with the property that $\sup_{n\in Q}|\wh{\sigma}(n)-1|<\varepsilon$. This stands in contradiction with assumption (2) of Theorem \ref{Theorem 2}, and terminates the proof.

\section{Equidistribution and Kazhdan sets in $\Z$}\label{Section 7}

As a straightforward consequence of Theorem \ref{Theorem 2}, we obtain the following answer to Question \ref{Que4}:

\begin{theorem}\label{The12}
 Let $(n_{k})_{\kk}$ be a sequence of elements of $\Z$ with $n_{0}=1$. Suppose that the sequence $(n_{k}\theta )_{\kk}$ is uniformly distributed modulo $1$ for every $\theta \in\R\setminus \Q$. Then $\{n_{k}\,;\,\kk\}$ is a \ka\ set in $\Z$.
\end{theorem}

\begin{proof}
The proof uses the same idea as that of Example \ref{Exa1}.
 Let  $0<\varepsilon <1$, and suppose that $\sigma $ is probability measure on $\T$ such that $\sup_{\kk}|\wh{\sigma }(n_{k})-1|<\varepsilon $. Then 
\[
\biggl|\,\dfrac{1}{N}\sum_{k=0}^{N-1}\wh{\sigma }(n_{k})-1\,\biggr|<\varepsilon \quad\textrm{for every}\ N\ge 1.
\]
Our assumption that $(n_{k}\theta )_{\kk}$ is uniformly distributed modulo $1$ for every $\theta \in \R\setminus \Q$ means that 
\[
\dfrac{1}{N}\sum_{k=0}^{N-1}\lambda ^{n_{k}}\!\xymatrix@C=22pt{\ar[r]&}\!0\ \textrm{as}\  N\!\xymatrix@C=17pt{\ar[r]&}\!+\infty\quad \textrm{for every}\ \lambda =e^{2i\pi \theta }\ \textrm{with}\ \theta \in \R\setminus\Q.
\] 
If we denote by $C$ the set of all roots of unity, this implies that 
\[
\sigma(C)\ge
\underline{\vphantom{p}\lim}_{\,N\to+\infty }\biggl|\,\int_{C}\biggl( \dfrac{1}{N}\sum_{k=0}^{N-1}\lambda ^{n_{k}}\biggr)d\sigma (\lambda )\,\biggr|\ge1-\varepsilon >0,
\]
and thus in particular $\sigma (C)>0$. Hence $\sigma $ has a discrete part, and Theorem \ref{Theorem 2}
 implies that $\{n_{k}\,;\,\kk\}$ is a \ka\ set in $\Z$.
\end{proof}

More generally, we have:

\begin{theorem}\label{The14}
 Let $(n_{k})_{\kk}$ be a sequence of elements of $\Z$ such that $(n_{k}\theta )_{\kk}$ is uniformly distributed modulo 1 for every $\theta \in\R\setminus \Q$. Then $\{n_{k}\,;\,\kk\}$ is a \ka\ set in $\Z$ if and only if it generates $\Z$.
\end{theorem}

\begin{proof}
As observed already in Example \ref{Exa2}, sets of the form $a\Z$, where $a\ge 2$ is an integer, are not \ka\ sets in $\Z$, while they satisfy the assumption that $(n_{k}\theta )_{\kk}$ is uniformly distributed modulo $1$ for every $\theta \in \R\setminus \Q$.
 So we only have to prove that if $\{n_{k}\,;\,\kk\}$ generates $\Z$ and satisfies the equidistribution assumption of Theorem \ref{The14}, $\{n_{k}\,;\,\kk\}$ must be a \ka\ set in $\Z$. Suppose it is not the case. There exist then $r\ge 2$ and $k_{1},\dots,k_{r}\ge 0$  such that $n_{k_{1}},\dots,n_{k_{r}}$ have no non-trivial common divisor, and hence integers $a_{1},\dots,a_{r}$ in $\Z$ such that $a_{1}n_{k_{1}}+\cdots+a_{r}n_{k_{r}}=1$. Since $\{n_{k}\,;\,\kk\}$  is not a \ka\ set in $\Z$,  there exists for every $\varepsilon >0$ a probability measure $\sigma $ on $\T$ with $\sigma ({\{1\}})=0$ such that $\sup_{\kk}|\wh{\sigma }(n_{k})-1|<\varepsilon $. Then 
\[
|\wh{\sigma }(1)-1|\le\sum_{i=1}^{r}|a_{i}|\,|\wh{\sigma }(n_{k_{i}})-1|<\biggl( \sum_{i=1}^{r}|a_{i}|\biggr)\,\varepsilon .
\]
This being true for all $\varepsilon >0$, it implies that the sequence $(p_{k})_{k\ge 0}$ defined by the conditions 
$p_{0}=1$ and $\{p_{k}\,;\,\kk\}=\{n_{k}\,;\,\kk\}\cup\{1\}$ is such that $\{p_{k}\,;\,\kk\}$ is a not \ka\ set in $\Z$. But $(p_{k}\theta )_{\kk}$ is still uniformly distributed modulo $1$ for every $\theta \in \R\setminus \Q$, and this contradicts Theorem \ref{The12}.
\end{proof}

As a direct corollary of Theorem \ref{The14}, we obtain:

\begin{example}\label{Exa15}
 Let $p\in\Z[X]$ be a non-constant polynomial. If the integers $p(k)$ have no non-trivial common divisor, $\{p(k)\,;\,k\ge 0\}$ is a \ka\ set in $\Z$.
\end{example}

\begin{proof}
 It is well-known (see for instance \cite{KN}) that for any non-constant polynomial $p\in\Z[X]$ the sequence $(p(k)\theta )_{k\ge 0}$ is uniformly distributed modulo $1$ for every $\theta \in \R\setminus \Q$. So Theorem \ref{The14} applies.
\end{proof}

We finish this section with a remark concerning finite perturbations of \ka\ sets.

\begin{proposition}\label{Pro16}
 Let $Q$ be a subset of $\Z$ which generates $\Z$. If $Q$ is not a \ka\ set in $\Z$, $Q\cup F$ is not a \ka\ set in $\Z$ either, whatever the choice of the finite subset $F$ of $\Z$.
\end{proposition}

Proposition \ref{Pro16} applies in particular when $1$ belongs to $Q$. It obviously breaks down if one discards the assumption that  $Q$ generates $\Z$: the set $2\N$ is not a \ka\ set in $\Z$, while $2\N\cup\{1\}$ is a \ka\ set.

\begin{proof}
 The proof is essentially the same as that of Theorem \ref{The14}. Under the assumptions of Proposition \ref{Pro16}, there exists for every $\varepsilon >0$ a probability measure $\sigma $ on $\T$ with $\sigma (\{1\})=0$ such that $\sup_{n\in Q}|\wh{\sigma }(n)-1|<\varepsilon $ and $|\wh{\sigma }(1)-1|<\varepsilon $. If follows that $|\wh{\sigma }(n)-1|<\#F\,\varepsilon $ for every $n\in F$. This being true for all $\varepsilon >0$, $Q\cup F$ is not a \ka\ set in $\Z$.
\end{proof}

\begin{corollary}
Let $Q$ be a \ka\ set in $\Z$ such that $1$ belongs to $Q$, and let $F$ be a finite subset of $Q$, not containing the point $1$. Then $Q\setminus F$ is still a \ka\ set in $\Z$.
\end{corollary}

\end{document}